\documentclass[11pt,a4paper,reqno]{amsart}
\usepackage[czech,english]{babel}
\usepackage{amscd,amssymb}
\usepackage{dsfont}
\usepackage{url}
\usepackage[all]{xy}
\title[Smashing localizations of rings of weak dimension one]%
{Smashing localizations of rings of weak global dimension at most one}
\subjclass[2010]{Primary: 13D09, 18E35; Secondary: 16E45}
\keywords{Smashing localization, homological epimorphism, Telescope Conjecture}

\date{\today}
\thanks{First named author supported by grant CPDA105885/10 of Padova University ``Differential graded categories'' and by Progetto di Eccellenza Fondazione Cariparo ``Algebraic Structures and their applications''. Second named author supported by grant GA~\v{C}R P201/12/G028 ``Eduard \v{C}ech Institute for Algebra, Geometry and Physics'' from Czech Science Foundation.}

\author[S. Bazzoni]{Silvana Bazzoni}
\address[Silvana Bazzoni]{%
Dipartimento di Matematica \\
Universit\`a di Padova \\
via Trieste 63, I-35121 Padova - Italy}
\email{bazzoni@math.unipd.it}

\author[J. \v{S}\v{t}ov\'{\i}\v{c}ek]{Jan \v{S}\v{t}ov\'{\i}\v{c}ek}
\address[Jan \v{S}\v{t}ov\'{\i}\v{c}ek]{%
Charles University in Prague, Faculty of Mathematics and Physics \\
Department of Algebra \\
Sokolovsk\'a 83, 186 75 Praha 8, Czech Republic
}
\email{stovicek@karlin.mff.cuni.cz}

\renewcommand{\iff}{if and only if }
\newcommand{\st}{such that }
\newcommand{\wrt}{with respect to }

\newcommand{\la}{\longrightarrow}
\newcommand{\mapr}[1]{\overset{#1}\longrightarrow}

\newcommand{\dd}{\colon}

\newcommand{\ph}{\varphi}

\newcommand{\dif}{\partial}

\newcommand{\wfs}{weak factorization system }
\newcommand{\wfss}{weak factorization systems }

\newcommand{\inv}{^{-1}}

\newcommand{\bbN}{\mathbb{N}}
\newcommand{\bbZ}{\mathbb{Z}}
\newcommand{\bbQ}{\mathbb{Q}}

\newcommand{\ifr}{\mathfrak{i}}
\newcommand{\jfr}{\mathfrak{j}}
\newcommand{\m}{\mathfrak{m}}
\newcommand{\n}{\mathfrak{n}}
\newcommand{\p}{\mathfrak{p}}
\newcommand{\q}{\mathfrak{q}}

\newcommand{\Spec}{\operatorname{Spec}} 
\newcommand{\Max}{\operatorname{Max}}    % Zariski spectrum
\newcommand{\iSpec}{\operatorname{iSpec}}  % idempotent spectrum of a valuation domain (paper specific)
\newcommand{\Inter}{\operatorname{Inter}}  % admissible intervals in the specture (paper specific)
\newcommand{\Hom}{\operatorname{Hom}}
\newcommand{\HOM}{\mathcal{H}\mathnormal{om}}
\newcommand{\End}{\operatorname{End}}

\newcommand{\Ext}{\operatorname{Ext}}
\newcommand{\Tor}{\operatorname{Tor}}
\newcommand{\Ker}{\operatorname{Ker}}
\newcommand{\Img}{\operatorname{Im}}

\newcommand{\wgldim}{\operatorname{w.gl.dim}}

\newcommand{\unit}{\mathds{1}}
\newcommand{\Supp}{\mathrm{Supp}\,}

\newcommand{\A}{\mathcal{A}}

\newcommand{\C}{\mathcal{C}}
\newcommand{\D}{\mathcal{D}}

\newcommand{\I}{\mathcal{I}}
\newcommand{\clL}{\mathcal{L}}

\newcommand{\R}{\mathcal{R}}
\newcommand{\clS}{\mathcal{S}}
\newcommand{\T}{\mathcal{T}}

\newcommand{\X}{\mathcal{X}} 
\newcommand{\Y}{\mathcal{Y}}
\newcommand{\Z}{\mathcal{Z}}  

\newcommand{\Modr}[1]{\mathrm{Mod}\textrm{-}{#1}}

\newcommand{\ModR}{\mathrm{Mod}\textrm{-}R}

\newcommand{\modR}{\mathrm{mod}\textrm{-}R}

\newcommand{\Dga}[1]{\mathrm{Dga}({#1})}

\newcommand{\Ab}{\mathrm{Ab}}

\newcommand{\op}{^\textrm{op}}

\newcommand{\Der}[1]{\mathbf{D}({#1})}

\newcommand{\Cpx}[1]{\mathbf{C}({#1})}

\newcommand{\cof}{\mathrm{Cof}}

\newcommand{\we}{\mathrm{W}}

\newcommand{\fib}{\mathrm{Fib}}

\newcommand{\Ho}{\operatorname{Ho}}

\newcommand{\Rder}[1]{\mathbf{R}{#1}}
\newcommand{\Lotimes}{\otimes^{\mathbf{L}}}
\newcommand{\RHom}{\Rder\Hom}

\theoremstyle{plain}
\newtheorem{thm}{Theorem}[section]
\newtheorem{lem}[thm]{Lemma}
\newtheorem{prop}[thm]{Proposition}
\newtheorem{cor}[thm]{Corollary}

\theoremstyle{definition}
\newtheorem{defn}[thm]{Definition}
\newtheorem{constr}[thm]{Construction}
\newtheorem{nota}[thm]{Notation}
\theoremstyle{remark}
\newtheorem{rem}[thm]{Remark}

\newtheorem{expl}[thm]{Example}

\begin{document}

\begin{abstract}
We show for a ring $R$ of weak global dimension at most one that there is a bijection between the smashing subcategories of its derived category and the equivalence classes of homological epimorphisms starting in $R$. If, moreover, $R$ is commutative, we prove that the compactly generated localizing subcategories correspond precisely to flat epimorphisms. We also classify smashing localizations of the derived category of any valuation domain, and provide an easy criterion for the Telescope Conjecture (TC) for any commutative ring of weak global dimension at most one. As a consequence, we show that the TC holds for any commutative von Neumann regular ring $R$, and it holds precisely for those Pr\"ufer domains which are strongly discrete.
\end{abstract}

\maketitle

% -----------------------------------------------------------------------------
% The table of contents
\setcounter{tocdepth}{1}
\tableofcontents

% ------------------------------------------------------------------------------
% The document body
\section*{Introduction}

If $R$ is a ring and $\Der R$ its unbounded derived category, it is usually hopeless to try to understand all objects of $\Der R$. A fruitful and recently extensively studied approach is to try to understand the inner structure of $\Der R$ through various localizations of $\Der R$. As demonstrated by our present paper and also elsewhere, triangulated localization theory provides a fascinating natural meeting point for abstract homotopy theory, algebraic geometry, homological algebra, module theory and other fields.

However, only compactly generated localizations of $\Der R$ with $R$ commutative are well understood in general. In fact, the classification started with results by Devinatz, Hopkins and Smith~\cite{DHS88}, and Neeman~\cite{Nee92}, and was finished by Thomason~\cite{Th97}. These results have been recently considerably extended and further interesting applications found by Balmer~\cite{Bal05} and Benson, Iyengar and Krause~\cite{BIK08,BIK11}. For more general localizations, the situation remains not so clear. To understand all Bousfield localizations of $\Der R$ is generally an extremely difficult problem as illustrated in~\cite{Nee00-oddball,DwPa08,Ste14}.

However, there is an intermediate class of so-called smashing localizations---those where the localization functor is given by tensoring. In contrast to the present state of art in stable homotopy theory, in the case of derived categories of rings of weak global dimension $\le 1$ (that is, rings whose all modules have flat dimension $\le 1$) this is a perfectly tractable class. One of our main results is a complete classification of smashing localizations of $\Der R$ for a valuation domain $R$. This seems to give one of a very few positive results for non-compactly generated localizations of $\Der R$ with $R$ non-noetherian.

Of course, smashing localizations are also intimately related to the Telescope Conjecture from the works of Bousfield and Ravenel~\cite{Bou79,Rav84}. The conjecture asks whether every smashing localization is compactly generated. In fact, it makes more sense to ask whether a particular triangulated category satisfies the Telescope Conjecture as there are derived categories which do not have this property~\cite{Kel94-smash}. Although in the original setting, for the stable homotopy category, the answer seems still unclear, for $\Der R$ with $\wgldim R \le 1$ we are sometimes even able to provide a list of all smashing localizations which are not compactly generated. Our hope is that this new light shed on the problem will foster further research and in the end leads to better understanding of triangulated localizations.

\smallskip

Let us briefly list the highlights of the present paper.

\begin{enumerate}
\item In Theorem~\ref{thm:smashing-classif} we explain the reason for the assumption of weak global dimension $\le 1$. In general, smashing localizations of $\Der R$ for $R$ not necessarily commutative are in bijection with equivalence classes of homological epimorphisms in the homotopy category of dg algebras. If $\wgldim R \le 1$, it suffices to study classical homological epimorphisms of rings. This often allows one to study smashing localizations in the module rather than in the derived category.

\item If, moreover, $R$ is commutative, we will show in Theorem~\ref{thm:flat-epi} that compactly generated localizations correspond precisely to flat ring epimorphisms $f\dd R \to S$.

\item In Theorem~\ref{thm:hom-epi-vd} we classify all smashing localizations of $\Der R$ with $R$ a valuation domain. We will show that knowing $\Spec R$ as a topological space is in general not enough to determine the lattice of smashing localizations, but knowing in addition which prime ideals are idempotent suffices. In particular, we immediately see which of the localizations are flat and whether the Telescope Conjecture holds for $\Der R$.

\item For commutative rings $R$ of weak global dimension $\le 1$ we are able to combine (2) and (3) in Theorem~\ref{thm:tc} to get a simple criterion for the Telescope Conjecture for $\Der R$. In particular we show that the conjecture holds for any commutative von Neumann regular ring $R$, generalizing a result from~\cite{Ste14}.
\end{enumerate}

% --------------------------------------
\subsection*{Acknowledgments}

The second named author would like to thank Pavel P\v{r}\'\i{}hoda for many interesting discussions, during which we were among others able to prove Theorem~\ref{thm:hom-epi-vd} for valuation domains with finite Zariski spectrum using a completely different method than the one presented here. Although unfortunately none of this was in the end used in the present text, his help was very important to finish this work.

% ------------------------------------------------------------------------------
\section{Smashing localization of triangulated categories}
\label{sec:smash}

Let $\T$ be a triangulated category with small coproducts and let us denote the suspension functor by $\Sigma$. We refer to~\cite{Nee01} for the definitions and abstract theory.

Smashing localizations and smashing subcategories arise naturally if, moreover, $\T$ admits a closed symmetric monoidal structure which is compatible with the triangulated structure in the sense of~\cite[Appendix A.2]{HPS97}. Such categories are also called tensor triangulated; see~\cite{BaFa11}. This situation arises in particular when one considers the stable homotopy category of spectra with the smash product, or the derived category $\Der R$ of a commutative ring together with the usual derived tensor product $\Lotimes_R$. 

Here, however, we shall mostly focus on a different branch of the theory of smashing localizations which does not require any monoidal structure on $\T$. Our main references are Pauksztello~\cite{Pauk09} and Nicol\'as and Saor\'\i{}n~\cite{NS09}.

% --------------------------------------
\subsection{Bousfield localization}
\label{subsec:Bousfield}

When $\T$ is a triangulated category with coproducts, we are often interested only in Verdier quotients $\T/\X$ (see~\cite[Ch. 2]{Nee01}) \st the canonical functor $q^*\dd \T \to \T/\X$ preserves coproducts. Equivalently, the class $\X$ is localizing in the following sense.

\begin{defn} \label{def:localizing}
A full subcategory $\X$ of $\T$ is \emph{localizing} if it is triangulated and closed under set indexed coproducts.
\end{defn}

Note that such an $\X$ is automatically closed under direct summands by~\cite[1.6.8]{Nee01}. In this situation it often happens that $q^*$ has a right adjoint $q_*\dd \T/\X \to \T$. In fact, the existence of $q_*$ is often equivalent to the fact that $\T/\X$ has small homomorphism spaces, so that it is a category in the usual sense (see for instance~\cite[Lemma 3.5]{Kr00} or~\cite[Proposition 9.1.19]{Nee01}).

The existence of the right adjoint $q_*$ to the localization functor $q^*$ is equivalent to the existence of a right adjoint functor $i^!$ to the inclusion functor $i_*\dd \X\to \T$. This is further equivalent to the existence of a so-called Bousfield localization functor $L\colon \T\to \T$ such that $\Ker L=\X$; see for instance~\cite[Proposition 4.9.1]{Kr10}.

\begin{defn} \label{def:Bousfield}
A \emph{Bousfield localization functor} is a triangulated endofunctor $L\colon \T\to \T$ together with a natural transformation $\eta:Id_{\T}\la L$ with $L\eta :L \la L^2$ being invertible and $L\eta = \eta L$. The objects in the essential image of $L$ are called \emph{$L$-local} and the objects in $\Ker L$ are called \emph{$L$-acyclic}.
\end{defn}

A very convenient way to describe a Bousfield localization functor $L$ is via a triangulated analogue of a torsion pair. If $L\dd \T \to \T$ is such a functor, then the pair $(\X,\Y) = (\Ker L,\Img L)$ of full subcategories of $\T$ enjoys the following properties:

\begin{enumerate}
\item $\X = \Sigma\X$ and $\Sigma \Y = \Y$;
\item $\T(X,Y) = 0$ for all $X \in \X$ and $Y \in \Y$;
\item For each $W \in \T$, there is a triangle of the form
\[ X \la W \la Y \la \Sigma X \]
with $X \in \X$ and $Y \in \Y$.
\end{enumerate}

It is well known that the map $X \to W$ in a triangle as in (3) is always an $\X$-coreflection, and the map $W \to Y$ is a $\Y$-reflection. Moreover, $\X$ and $\Y$ are triangulated subcategories of $\T$ and determine each other: $\X = {^\perp\Y}$ and $\Y = \X^\perp$. Here we use the following notation for so-called \emph{perpendicular classes} to a class of objects $\C \subseteq \T$:
\begin{align*}
\C^{\perp} &=
\{X\in \T\mid \Hom_{\T}(\Sigma^nC,X)=0 \mathrm{\ for\ all\ } C\in \C \mathrm{\ nad\ }n\in \bbZ \},
\\
^{\perp }\C &=
\{X\in \T\mid \Hom_{\T}(X,\Sigma^nC)=0 \mathrm{\ for\ all\ } C\in \C \mathrm{\ nad\ }n\in \bbZ \}.
\end{align*}

It is also a standard fact (see for instance~\cite[Ch. 9]{Nee01} or~\cite[\S4.9]{Kr10}) that on the other hand a pair $(\X,\Y)$ satisfying (1)--(3) determines a Bousfield localization functor $L$ with $\Ker L = \X$, and such an $L$ is unique up to a suitably defined natural equivalence. More precisely, if $L'$ is another such functor together with $\eta':Id_{\T}\la L'$, then there is a unique natural equivalence $\xi\dd L \to L'$ \st $\eta' = \xi\eta$.

Needless to say that this setup has been observed in other situations. Bondal and Orlov~\cite{BO95} call such $(\X,\Y)$ semiorthogonal decompositions of $\T$. The pair $(\X,\Y)$ can also be viewed as a $t$-structure with a trivial heart in the sense of Beilinson, Bernstein and Deligne \cite{BBD}.

% --------------------------------------
\subsection{Smashing localization and TTF triples}
\label{subsec:TTF-triples}

If $\T$ admits a closed symmetric monoidal structure $(\T,\otimes,\unit)$, one may ask under which conditions a Bousfield localization functor is equivalent to $- \otimes Y$ for some $Y \in \T$. This in particular happens for localizations generated by a small set of compact objects~\cite[Theorem 3.3.3]{HPS97}:

\begin{defn} \label{def:comp-gen-loc}
An object $C \in \T$ of an additive category with arbitrary small coproducts is called \emph{compact}, if the natural homomorphism
\[ \coprod_{i \in I} \T(C,Y_i) \la \T(C, \coprod_{i \in I} Y_i) \]
is an isomorphism for any small collection $(Y_i \mid i \in I)$ of objects of $\T$.

A Bousfield localization functor is called \emph{compactly generated} if there is a small set $\C \subseteq \T$ of compact objects such that the class of $L$-local objects is equal to $\C^\perp$. Equivalently we may require that the set of $L$-acyclic objects is the smallest localizing subcategory of $\T$ containing~$\C$; see~\cite[Lemma 1.7]{Nee92-loc}.% and \S\ref{subsec:Bousfield}.
\end{defn}

One feature of a localization functor of the form $L = - \otimes Y$ is that $L$ preserves coproducts in $\T$. As the latter property rather often characterizes localizations of the form $L = - \otimes Y$ (see~\cite[Definition 3.3.2]{HPS97}), it was taken by Krause~\cite{Kr00,Kr05-telescope} as the definition of a smashing localization in the absence of a tensor product:

\begin{defn} \label{def:smashing}
A Bousfield localization functor $L\dd \T \to \T$ is called \emph{smashing} if it preserves coproducts. A localizing class $\X \subseteq \T$ is called \emph{smashing} if it is the class of acyclic objects for a smashing localization functor.
\end{defn}

If we do not wish to refer to the localization functor explicitly, we can use the following lemma:

\begin{lem} \label{lem:smash-coprod}
Let $\X \subseteq \T$ be a localizing subcategory. Then the following are equivalent:
\begin{enumerate}
\item $\X$ is smashing.
\item The inclusion functor $i_*\dd \X \to \T$ admits a right adjoint $i^!$ and the perpendicular class $\X^\perp$ is closed under small coproducts.
\end{enumerate}
\end{lem}

\begin{proof}
See the argument in~\cite[Definition 3.3.2]{HPS97}.
\end{proof}

If now $L$ is a smashing localizing functor and $\Y = \Img L$ is the class of $L$-local objects, it is again a localizing class. This suggests that there should exist another Bousfield localization $L'\dd \T \to \T$ \st $\Y = \Ker L'$. This is indeed often the case, assuming a technical condition on $\T$.
A sufficient one is a so called Brown representability condition. That is, we require that every cohomological functor $F\dd \T\op \to \Ab$ (in the sense of~\cite[Remark 1.1.9]{Nee01}) which preserves small products is isomorphic to $\T(-,E)$ for an object $E \in \T$.
Note that any compactly generated or well generated triangulated category in the sense of~\cite{Nee01,Kr10} has this property. In particular, the unbounded derived category $\Der R$ of any ring $R$ (commutative or not), or more generally of any dg algebra $R$, is an example, see~\cite[\S8.1.3]{Kel98} and \cite[\S4.2 and \S5.2]{Kel94}.

Now we can give the characterization, which closely relates smashing localizations to recollements~\cite[\S4.13]{Kr10} (see also Remark~\ref{rem:recollement}).

\begin{defn} \label{def:ttf}
A \emph{torsion-torsion-free triple} (\emph{TTF triple} for short) on a triangulated category $\T$ is a triple $(\X, \Y, \Z)$ of full subcategories of $\T$ such that both $(\X,\Y)$ and $(\Y,\Z)$ enjoy properties (1)--(3) stated in~\S\ref{subsec:Bousfield}. Equivalently, $(\X,\Y)$ and $(\Y,\Z)$ both determine $t$-structures on $\T$ in the sense of~\cite{BBD}.
\end{defn}

\begin{prop} \label{prop:recollement}
Let $\T$ be a triangulated category with small coproducts \st every cohomological functor which preserves small products is representable. Let $L\dd \T \to \T$ be a Bousfield localization functor and $\X = \Ker L$. Then the following are equivalent:
\begin{enumerate}
\item $\X$ is smashing;
\item There is a TTF triple $(\X, \Y, \Z)$ on $\T$.
\end{enumerate}
\end{prop}

\begin{proof}
See the proof of~\cite[Proposition 5.5.1]{Kr10}.
\end{proof}

% ------------------------------------------------------------------------------
\section{Homological epimorphisms for dg algebras}
\label{sec:hom-epi-dga}

Next we shall consider smashing localizations at the level of model categories, which for us always means dg algebras and dg modules over them (see Appendix~\ref{sec:homotopy-dg} for details). The advantage is that building on Pauksztello's definition of a homological epimorphism of dg algebras and the results in~\cite{NS09}, a smashing localization turns out to be always given by a certain tensor product, even if the triangulated category in question is not tensor triangulated. Later on we will show that for derived categories of rings of weak global dimension at most one everything simplifies to the notion of classical (homological) epimorphisms of ordinary rings.

To start with, note that every morphism $A \to C$ in the homotopy category $\Ho\Dga k$ of dg algebras over $k$ is by Remark~\ref{rem:morphisms-in-model-cat} and Proposition~\ref{prop:model-dgas} represented by a fraction
\[
\xymatrix{
& B \ar[dl]_\sigma \ar[dr]^f \\
A && C,
}
\]
where $\sigma\dd B \to A$ is a surjective quasi-isomorphism of dg algebras and $B$ is cofibrant. Generalizing Pauksztello's~\cite[Definition 3.10]{Pauk09}, we say that:

\begin{defn} \label{def:hom-epi-dga}
A morphism $f\sigma\inv\dd A \to C$ in $\Ho\Dga k$ (with $\sigma\dd B \to A$ and $f\dd B \to C$ as above) is a \emph{homological epimorphism} if the canonical map $C \Lotimes_B C \to C$ coming from the multiplication morphism is a quasi-isomorphism.
%
%Two homological epimorphisms $g = f\sigma\inv\dd A \to C$ and $g' = f'(\sigma')\inv\dd A \to C'$ are equivalent if there exists an isomorphism $\ph\dd C \to C'$ in $\Ho\Dga k$ \st $g' = \ph g$.
\end{defn}

\begin{rem} \label{rem:context}
The symbol $C \Lotimes_B C$ may cause some confusion since $- \otimes_B -$ may be viewed as a functor $\Cpx B \times \Cpx{B\op} \to \Cpx k$, but also $\Cpx B \times \Cpx{B\op \otimes_k C} \to \Cpx C$ and in other similar contexts. Firstly, it turns out that it actually does not matter for Definition~\ref{def:hom-epi-dga} which of these functors we derive. Secondly, we need that $C \Lotimes_B C \to C$ is a quasi-isomorphism when we view $- \otimes_B -$ as a functor $\Cpx B \times \Cpx{B\op \otimes_k C} \to \Cpx C$ because this is the interpretation in~\cite{NS09} and yields crucial Proposition~\ref{prop:descr-dga}.
\end{rem}

Now we need to prove that our notion of a homological epimorphism is well defined in the following sense:

\begin{lem} \label{lem:hom-epi-dga}
The definition of a homological epimorphism in $\Ho\Dga k$ is independent of the particular choice of the representing fraction $f\sigma\inv$.
\end{lem}

\begin{proof}
If $f'(\sigma')\inv$ is another fraction \st $\sigma'\colon B' \to A$ is a surjection from a cofibrant dg algebra and $f\sigma\inv = f'(\sigma')\inv$ in $\Ho\Dga k$,
%
%\[
%\xymatrix{
%& B' \ar[dl]_{\sigma'} \ar[dr]^{f'} \\
%A && C,
%}
%\]
%
there is a quasi-isomorphism $\sigma''\dd B \to B'$ of dg algebras \st $\sigma = \sigma'\sigma''$ and $f \sim f'\sigma''$ (see~Remark~\ref{rem:morphisms-in-model-cat}). That is, there is a cylinder object (Definition~\ref{def:htpy})
\[
B \amalg B \mapr{i_0+i_1} D \mapr{\tau} B
\]
and a map $h\dd D \to C$ \st $hi_0 = f$ and $hi_1 = f'\sigma''$.

Since $\tau i_0=1_B=\tau i_1$ by definition, the 2-out-of-3 property implies that also $i_0, i_1$ are quasi-isomorphisms. Thus,  Lemma~\ref{lem:dga-change} provides us with isomorphisms in $\Der k$:
\[ C \Lotimes_{B'} C \cong C \Lotimes_B C \cong C \Lotimes_D C \cong C \Lotimes_B C. \]
Here, $C$ is viewed as a dg $D$-bimodule via the map $h$.
Our final comment is regarding the double occurrence of $C \Lotimes_B C$ in the chain of isomorphisms---this is because the first copy is taken \wrt the morphism $f'\sigma''\dd B \to C$, while the second one is \wrt $f\dd B \to C$.
\end{proof}

\begin{rem} \label{rem:recollement}
Note that if $f\sigma\inv\dd A \to C$ represents a homomorphism in $\Ho\Dga k$, then the quasi-isomorphism $\sigma\dd B \to A$ induces a triangle equivalence $\sigma_*\dd \Der A \to \Der B$, whose quasi-inverse is $\sigma^* = -\Lotimes_B A\dd \Der B \to \Der A$ (see \cite[Lemma 6.1~(a)]{Kel94}). Hence we have a triangle functor
\[ \xymatrix@1{\Der C \ar[rrr]^{\sigma^*f_*} &&& \Der A}, \]
which takes the role of the functor induced by the restriction of scalars. Then~\cite[Theorem 3.9]{Pauk09} says that $\sigma^*f_*$ is fully faithful \iff $f\sigma\inv$ is a homological epimorphism in the sense of Definition~\ref{def:hom-epi-dga}.

It is not difficult to convince oneself that then $-\Lotimes_B C$ and $\RHom_{B}(C,-)$ are, respectively, left and right adjoint of the functor $\sigma^*f_*$. The situation can be described by the following diagram:  
\[
\xymatrix{
\Der C  \ar[rrr]^{\sigma^*f_*}
&&&\Der A
\ar@/^2pc/ [lll]^{\mathbb \RHom_{B}(C,-)}
\ar@/_2pc/ [lll]_{- \Lotimes_B C}
}
\]
Denoting $\X = \Ker (- \Lotimes_B C)$, $\Y = \Img \sigma^*f_*$ and $\Z = \Ker\big(\RHom_B(C,-)\big)$, where $\Ker$ and $\Img$ stand for the kernel on objects and the essential image, respectively, we obtain a torsion-torsion-free triple, that is a triple $(\X,\Y,\Z)$ in $\Der A$ as in Definition~\ref{def:ttf}.

We will also need to understand the corresponding localization functor $(L,\eta)$ (see Definition~\ref{def:Bousfield}) such that $\Ker L = \X$. Following the recipe in~\cite[\S4.9]{Kr10}, we take $L = \sigma^*f_*(-\Lotimes_B C)$ and for $\eta$ the unit of the adjunction. Since $f_*$ is just a forgetful functor and $\sigma^* = - \Lotimes_B A$, we have $L \cong - \Lotimes_B (C \Lotimes_B A)$. Since $C \Lotimes_B A \Lotimes_B C \cong C \Lotimes_B B \Lotimes_B C \cong C \Lotimes_B C \cong C$, we indeed get that $L \cong L^2$ via $\eta$.
\end{rem}

Following a suggestion of Pedro Nicol\'as, we now provide a slight improvement of the main result of~\cite{NS09} which basically says that the converse of the latter observation is true. That is, every torsion-torsion-free triple in $\Der A$ occurs in this way.

\begin{prop} \label{prop:descr-dga}
Let $A$ be a dg algebra over a commutative ring $k$ and $\X \subseteq \Der A$ be a smashing localizing class. Then there is a homological epimorphism $g = f\sigma\inv\dd A \to C$ in $\Ho\Dga k$, represented by homomorphisms of dg algebras $\sigma\dd B \to A$ and $f\dd B \to C$ as in Definition~\ref{def:hom-epi-dga}, \st
\[ \sigma^*f_*\dd \Der C \la \Der A. \]
is fully faithful, its essential image coincides with $\X^\perp$ and $\X = \{ X \in \Der A \mid X \Lotimes_B C = 0 \}$.

Conversely, if $g = f\sigma\inv\dd A \to C$ is a homological epimorphism in the category $\Ho\Dga k$, then $\X = \{ X \in \Der A \mid X \Lotimes_B C = 0 \}$ is a smashing localizing class in $\Der A$.
\end{prop}

\begin{proof}
Suppose that $\X \subseteq \Der A$ is smashing localizing and let $\sigma\dd B \to A$ be a cofibrant replacement of $A$. Then $B$ is homotopically projective in $\Cpx k$ by Proposition~\ref{prop:model-dgas} and the restriction functor $\sigma_*\dd \Der A \to \Der B$ is a triangle equivalence with $\sigma^* = - \Lotimes_B A$ as a quasi-inverse. Let now $\X'$ be the essential image of $\X$ under $\sigma_*$. It is a localizing subcategory of $\Der B$ and it is smashing since this property is preserved by equivalences of categories by Lemma~\ref{lem:smash-coprod}.
So by \cite[Theorem in \S4]{NS09} there exists a morphism of dg algebras $f\dd B \to C$ \st $C \Lotimes_B C \to C$ is a quasi-isomorphism and $\Img f_* = (\X')^\perp$. Hence $g = f\sigma\inv\dd A \to C$ is a homological epimorphism in $\Ho\Dga k$ and $\X^\perp$ is the essential image of the fully faithful functor $\sigma^*f_*\dd \Der C \la \Der A$. It also follows that  
\[ \X' = \{ X' \in \Der B \mid X' \Lotimes_B C = 0 \} \]
(see Remark~\ref{rem:recollement} and the recollement on p.~1232 in~\cite[\S4]{NS09}).
Transferring this along the triangle equivalence $\sigma^*$ provides the formula for $\X$.

The last part follows from Remark~\ref{rem:recollement} and~\cite[Theorem in \S5]{NS09}.
\end{proof}

Hence there is a surjective correspondence from the class of homological epimorphisms in $\Ho\Dga k$ originating in $A$ to the set (not a proper class, see~\cite{Kr00}) of smashing localizing classes in $\Der A$. One might ask when exactly two homological epimorphisms $g\dd A \to C$ and $g'\dd A \to C'$ induce the same smashing localizing class. The answer is given by the following result which we will prove in a special case in the next section.

\begin{prop} \label{conj:hom-epi-dga}\cite{NS13}
%\todo{Get a reference from Manolo and Pedro to their manuscript where this is solved.}
Two homological epimorphisms $g\dd A \to C$ and $g'\dd A \to C'$ induce the same smashing localizing class in $\Der A$ \iff there exists an isomorphism $\ph\dd C \to C'$ in $\Ho\Dga k$ \st $g' = \ph g$.
\end{prop}

% ------------------------------------------------------------------------------
\section{Homological epimorphisms for rings of weak dimension one}
\label{sec:hom-epi-semihered}

The main objects of interest in our paper are smashing localizations of the derived category $\Der R$ of a ring of weak global dimension at most one. As it turns out, the situation in this case is extremely favorable in that for studying smashing localizations of $\Der R$ it will be enough to consider homological epimorphisms of ordinary algebras (Definition~\ref{def:hom-epi-rings}) instead of homological epimorphisms of dg algebras (Definition~\ref{def:hom-epi-dga}). The aim of the section is to explain this reduction, which is already known for hereditary algebras~\cite{KrSt}.

% --------------------------------------
\subsection{Basics on homological epimorphisms of rings}
\label{subsec:hom-epi-rings}

We start by recalling some standard facts which we will need. Let $R,S$ be associative and unital algebras over a fixed base commutative ring $k$. This is no restriction at all since $k = \bbZ$ is a legal choice, but in some cases it may be convenient to take other base rings. 
We will denote by $\ModR$ and $\Modr{S}$ the categories of right $R$-modules and $S$-modules, respectively. An algebra homomorphism $f\dd R\to S$ is an epimorphism if it is an epimorphism in the category of $k$-algebras. Ring (and algebra) epimorphisms were investigated in \cite{Sil67,Sten75,GdlP87,Laz69}.

An algebra homomorphism $f\colon R\to S$ is an epimorphism \iff $S\otimes_RS\cong S$, \iff $1_R\otimes x=x\otimes 1_R$ in $S\otimes_RS$ for every $x\in S$, \iff the restriction functor $f_*\dd \Modr{S}\to \ModR$ is fully faithful (or the same holds for left modules). A direct way to present elements of $S$ in terms of elements of $R$, which is essentially due to Mazet~\cite{Maz68}, will be discussed later in \S\ref{subsec:lots-of-idemp}.

Two algebra epimorphisms $f\dd R \to S$ and $f'\dd R \to S'$ are said to be \emph{equivalent} if there exists a $k$-algebra isomorphism $\ph\dd S \to S'$ \st $f' = \ph f$. Equivalently, the essential images of $f_*$ and $f'_*$ in $\ModR$ coincide.

The following results will be useful in the sequel.

\begin{prop} \label{prop:silver-lazard}
Let $R$ be a commutative $k$-algebra and $f\dd R\to S$ an algebra homomorphism. The following hold true:
\begin{enumerate}
\item \cite[Corollary 1.2]{Sil67} If $f$ is an epimorphism, then $S$ is a commutative algebra.
\item \cite[Lemma 1.1]{Laz69} $f$ is an epimorphism \iff $f_\p\dd R_\p\to S\otimes_R R_\p$ is an epimorphism for every prime ideal $\p$ of $R$.
\end{enumerate}
\end{prop}

\begin{defn} \label{def:hom-epi-rings}
A $k$-algebra epimorphism $f\dd R\to S$ is a \emph{homological epimorphism} if $\Tor_i^R(S,S)=0$ for every $i\geq 1$.
\end{defn}

\emph{Homological algebra epimorphisms} have been introduced and characterized by Geigle and Lenzing in \cite{GL91}, see Proposition~\ref{prop:char-hom-epi} below. While an algebra epimorphism $R\to S$ implies that the category of $S$-modules is equivalent to a full subcategory of the category of $R$-modules, homological epimorphisms are characterized by the analogous property for derived categories.

An algebra epimorphism $f\dd R\to S$ with $S$ flat as a left or right $R$-module is clearly a homological epimorphism. It is called a \emph{flat epimorphism}.

\begin{prop} \label{prop:char-hom-epi} \cite[4.4]{GL91}
Let $R$, $S$ be $k$-algebras. An algebra homomorphism $f\dd R\to S$ is a homological ring epimorphism \iff one of the following equivalent conditions holds:
\begin{enumerate}
\item $S\otimes_RS\cong {_SS}_S$ and $\Tor_i^R(S,S)=0$ for every $i\geq 1$ (i.e.\ the natural map $S\Lotimes_RS\to S$ is an isomorphism).

\item For every right $S$-module $N$ and a left $S$-module $M$, the natural map $\Tor_i^R(N, M)\to \Tor_i^S(N, M)$ is an isomorphism for every $i\geq 0$ (i.e.\ the natural map $N \Lotimes_R M \to N \Lotimes_S M$ is an isomorphism).

\item For every $S$-modules $M, M'$, the natural morphism $\Ext_S^i(M,M')\to \Ext_R^i(M,M')$ is an isomorphism for every $i\geq 0$ (i.e.\ the natural morphism $\RHom_S(M,M') \to \RHom_R(M,M')$ is an isomorphism).

\item The induced functor $f_*\dd \Der S \la \Der R$ is a full embedding of triangulated categories.
\end{enumerate}
\end{prop}

In the coming lemma we collect some easy observations about homological epimorphisms.

\begin{lem}\label{lem:hom-epi-localization} ~
\begin{enumerate}
\item[(1)] The composition of homological epimorphisms is a homological epimorphism.
\end{enumerate}
If, moreover, $R$ is a commutative ring, then also the following hold:
\begin{enumerate} 
\item[(2)] If $\Sigma$ is a multiplicative subset of $R$, then $R\to R\Sigma^{-1}$ is a flat epimorphism.
\item[(3)] A ring homomorphism $f\dd R\to S$ is a homological epimorphism \iff $f_\p\dd R_\p\to S\otimes_R R_\p$ is a homological epimorphism for every prime ideal $\p\in \Spec R$.
\end{enumerate}
\end{lem}
\begin{proof} (1) Let $f\dd R\to S$ and $g\dd S\to T$ be two homological epimorphisms. Clearly $gf$ is an epimorphism and, since $T$ is an $S$-bimodule, $\Tor_i^R(T,T)\cong\Tor_i^S(T,T)=0$.

(2) It is well known that $R\Sigma^{-1}$ is flat and that $R\Sigma^{-1}\otimes_RR\Sigma^{-1}\cong R\Sigma^{-1}$.

(3) Let $i\geq 1$; then $\Tor_i^{R}(S,S)=0$ if and only if $\Tor_i^{R}(S,S)\otimes_R R_\p=0$ for every prime ideal $\p$ of $R$ (see the argument in~\cite[Definition 2.4.10]{EJ00}), and $\Tor_i^{R}(S, S)\otimes_R R_\p\cong \Tor_i^{R_\p}(S\otimes_R R_\p, S\otimes_R R_\p)$, for every prime ideal $\p$ of $R$, since $-\otimes_RR_\p$ is an exact functor (see also \cite[Theorem 2.1.11]{EJ00}). Thus the conclusion follows  by Proposition~\ref{prop:silver-lazard}(2).
\end{proof}

% --------------------------------------
\subsection{An application of K\"unneth's theorem}
\label{subsec:Kunneth}

Now we specialize to not necessarily commutative $k$-algebras $R$ of weak global dimension ($\wgldim$) at most $1$. That is, we require by definition that $\Tor_2^R(-,-) \equiv 0$, or equivalently that submodules of flat modules are flat. We start with collecting some easy facts about homological ring epimorphisms in this case. Compare also with \cite[Example 4]{NS09}.

\begin{lem} \label{lem:kernel-idempotent}
Let $R$ be a $k$-algebra with $\wgldim R\leq 1$ and let $f\dd R\to S$ be a homological epimorphism. Then the following hold:
\begin{enumerate}
\item $\Ker f$ is an idempotent two-sided ideal of $R$ and $\wgldim S\leq 1$.
\item The canonical projection $\pi\dd R \to R/\Ker f$ and the induced homomorphism $\overline{f}\dd R/\Ker f \to S$  are homological ring epimorphisms.
\end{enumerate}
Moreover, for any two sided ideal $I$, the canonical projection $R\to R/I$ is a homological ring epimorphism if and only if $I$ is an idempotent two-sided ideal of $R$.
\end{lem}

\begin{proof}
(1) Let $I=\Ker f$ and apply the functors $S\otimes_R-$ and $-\otimes_R R/I$ to the exact sequence
\[ 0 \la R/I \mapr{\overline{f}} S \la S/f(R) \la 0 \]
to get $0 = \Tor_2^R(S,S/f(R)) \to \Tor_1^R(S,R/I) \to \Tor_1^R(S,S)=0$ and $0=\Tor_2^R(S/f(R),R/I)\to \Tor_1^R(R/I,R/I)\to \Tor_1^R(S, R/I) = 0$. Consequently $\Tor_1^R(R/I,R/I)=0$. Now consider the exact sequence $0\to I\to R\to R/I\to 0$ and apply the functor $R/I\otimes_R-$ to obtain the exact sequence
\[ 0 \la R/I\otimes_R I \cong I/I^2 \la R/I \la R/I\otimes_RR/I \la 0, \] 
which yields $I^2=I$, since $R/I\otimes_RR/I\cong R/I$.

By Proposition~\ref{prop:char-hom-epi}(2), $\Tor_2^S(-,-)\cong \Tor_2^R(-,-)$, hence $\wgldim S\leq 1$.

(2) Let $I=\Ker f$. From the proof of part (1) $\Tor_1^R(R/I, R/I)=0$, thus, $\pi$ is a homological ring epimorphism. In particular, $\Tor_1^{R/I}(S,S)\cong \Tor^R_1(S,S)=0$, since $S$ is an $R/I$-bimodule. Moreover, $\overline{f}$ is clearly an algebra epimorphism, so also homological.

The previous arguments show that a two-sided ideal is idempotent if and only if $\Tor_1^R(R/I, R/I)=0$, hence the last statement follows immediately.
\end{proof}

In order to relate this to smashing localizations of $\Der R$ and homological epimorphisms of dg algebras, we state a version of K\"unneth's theorem.
   
\begin{prop}\label{prop:Kunneth}
Let $R$ be a $k$-algebra with $\wgldim R\leq 1$. Let $X$ be a complex of right $R$ modules and $Z$ a complex of left $R$-modules. Then, the following are equivalent:
\begin{enumerate}
\item $X \Lotimes_R Z = 0$ in $\Der\Ab$;
\item $\Tor_i^R\big(H^p(X), H^q(Z)\big)=0$ for every $p,q\in \bbZ$ and every $i\geq 0$;
\item $H^p(X) \Lotimes_R H^q(Z)=0$, for every $p,q\in \bbZ$.
\end{enumerate}
\end{prop}

\begin{proof}
Let $P \to X$ be a homotopically projective replacement of $X$ in $\Cpx R$ in the sense of \S\ref{subsec:model-dg-mod}, so that the morphism is a quasi-isomorphism and $P$ is homotopically projective. We have $X \Lotimes_R Z=0$ \iff $H^n(P\otimes_RZ)=0$ for every $n\in \bbZ$. The complex $P$ has projective terms, so the coboundary module $\dif(P^n)$, where $\dif\dd P^n \to P^{n+1}$ is the differential of $P$, is a flat submodule of $P^{n+1}$ for every $n\in \bbZ$. Similarly all cocycle modules of $P$ are flat. By K\"unneth's theorem~\cite[Ch. VI, Theorem 3.1]{CE56} there is an exact sequence:
\begin{multline} \label{eqn:Kunneth}
0 \la \bigoplus_{p+q=n} H^p(P) \otimes_R H^q(Z) \la H^n(P\otimes_RZ) \la \\ \la \bigoplus_{p+q=n+1} \Tor_1^R\big(H^p(P),H^q(Z)\big) \la 0.
\end{multline}

This establishes the equivalence (1) $\Leftrightarrow$ (2). The sequence (\ref{eqn:Kunneth}) considered for the complexes $H^q(Z)$ and $H^p(X)$ concentrated in degree zero gives the equivalence of conditions (2) and (3).
\end{proof}

%\begin{rem} \label{rem:essential-image}
%In order to better understand the relation between Remark~\ref{rem:recollement} and Proposition~\ref{prop:hom-epi-semihered}, suppose first that $f\dd R \to S$ is a ring epimorphism. Then essential image $\A$ of the induced functor $\Modr S \to \ModR$ is given by
%%
%\[ \A = \{ M \in \ModR \mid M \otimes_R f\dd M \to M \otimes_R S \textrm{ is an isomorphism} \}. \] 
%%
%If $f$ is a homological epimorphism, this can be improved to
%%
%\[ \A = \{ M \in \ModR \mid M \Lotimes_R f\dd M \to M \otimes_R S \textrm{ is an isomorphism} \}; \] 
%%
%see Proposition~\ref{prop:char-hom-epi}(2). In such a case, the essential image of the functor $f_* = -\Lotimes_S S_R\dd \Der S \to \Der R$ is equal to
%%
%\[ \Img f_* = \{ Y \in \Der R \mid Y \Lotimes_R f\dd Y \to Y \Lotimes_R S \textrm{ is an isomorphism} \}. \]
%\end{rem}

Suppose now that we have a homomorphism $f\sigma\inv\dd R \to C$ in $\Ho\Dga k$, assuming that $R$ is an algebra over a commutative base ring $k$. If $\sigma = 1_R$ (i.e.\ $f$ is represented by a morphism of dg algebras rather than a fraction), the last proposition says, using the notation of Remark~\ref{rem:recollement}, that
\begin{align*}
\X &= \Ker(-\Lotimes_R C) = \\ &= \big\{ X \in \Der R \mid \Tor_i^R\big(H^p(X), H^q(C)\big) = 0 \textrm{ for all } p,q \in \bbZ \textrm{ and } i \ge 0 \big\}.
\end{align*}
This is very convenient as it is enough to consider modules rather than complexes.

A complete analogy is true for general homomorphisms $f\sigma\inv$ starting at $R$, but more work is required. We first establish an auxiliary lemma, which is analogous to~\cite[Lemma 6.3]{Kel94}.

\begin{lem} \label{lem:representing-bimod}
Let $k$ be a commutative ring and $A$ be a dg algebra over $k$. Given a homomorphism $f\sigma\inv\dd A \to C$ in $\Ho\Dga k$, represented by homomorphisms of dg algebras $\sigma\dd B \to A$ and $f\dd B \to C$ as in Definition~\ref{def:hom-epi-dga}, \st $C$ is a cofibrant dg algebra in the sense of~Proposition~\ref{prop:model-dgas}, there exists a dg $A$-$C$-bimodule ${_AZ_C}$, homotopically projective as a left $A$-module, \st the functor
\[ -\otimes_A Z\dd \Der A \la \Der C \]
is naturally equivalent to $- \Lotimes_B C\dd \Der A \to \Der C$ (compare to Remark~\ref{rem:recollement}).
\end{lem}

\begin{proof}
Recall that $f\sigma\inv$ is a right fraction in $\Dga k$ of the form
\[
\xymatrix{
& B \ar[dl]_\sigma \ar[dr]^f \\
A && C,
}
\]
where $C$ is a cofibrant dg algebra and as such $C$ is homotopically projective in $\Cpx k$ by Proposition~\ref{prop:model-dgas}.

Let $v\dd {_BV_C} \to {_BC_C}$ be a homotopically projective resolution of $C$ in $\Cpx{B\op \otimes_k C}$. Since $C$ is homotopically projective in $\Cpx k$, it follows from Lemma~\ref{lem:htp-proj-tensor}(2) that $B\op \otimes_k C$ is homotopically projective as a left dg $B$-module. Applying Proposition~\ref{prop:htp-proj-dgmod} and Lemma~\ref{lem:htp-proj-tensor}(1), we deduce that any homotopically projective $B$-$C$-bimodule is a homotopically projective left dg $B$-module, and in particular so is ${_BV_C}$. Thus, if we put ${_AZ_C} = A \otimes_B V$, then $Z$ is homotopically projective in $\Cpx A$ again by Lemma~\ref{lem:htp-proj-tensor}(2) and we have the following natural isomorphisms in $\Der C$
\[ X \Lotimes_B C \cong X \otimes_B V \cong X \otimes_A (A \otimes_B V) = X \otimes_A Z \]
for each $X \in \Cpx A$.
\end{proof}

Now we can compute the kernel of $- \Lotimes_B C\dd \Der R \to \Der C$.

\begin{lem} \label{lem:kernel-of-localization}
Let $R$ be a $k$-algebra with $\wgldim R \le 1$ and let $f\sigma\inv\dd R \to C$ be a homomorphism in $\Ho\Dga k$. Using the notation of Remark~\ref{rem:recollement}, we put $\X = \Ker(-\Lotimes_B C) \subseteq \Der R$.

Given $X \in \Der R$, then $X \in \X$ \iff $H^p(X) \in \X$ for every $p \in \bbZ$ \iff $H^p(X) \Lotimes_R H^q(C) = 0$ for every $p,q \in \bbZ$.
\end{lem}

\begin{proof}
In order to make sense of the expression $H^p(X) \Lotimes_R H^q(C) = 0$, we inspect the fraction
\[
\xymatrix{
& B \ar[dl]_\sigma \ar[dr]^f \\
R && C.
}
\]
The quasi-isomorphism $\sigma$ induces an isomorphism of $k$-algebras $R \cong H^0(B)$ and each cohomology $H^p(C)$ is naturally an $H^0(B)$-module via $f$.

In order to prove the proposition, we may without loss of generality assume that $C$ is a cofibrant dg algebra over $k$. Indeed, otherwise we could take a trivial fibration $g\dd C' \to C$ in $\Dga k$ with $C'$ cofibrant and, $B$ being cofibrant, the map $f\dd B \to C$ would factor through $g$, keeping the class $\X$ unchanged.

After this reduction, we are in the situation of Lemma~\ref{lem:representing-bimod} and can interpret the functor $-\Lotimes_B C\dd \Der R \to \Der C$ as $-\otimes_R Z$ for a suitable dg bimodule ${_RZ_C}$ which is homotopically projective as a complex of left $R$-modules. Proposition~\ref{prop:Kunneth} now yields the equivalences
\begin{align*}
X \in \X & \quad \Longleftrightarrow \quad H^p(X) \in \X \textrm{ for every } p \in \bbZ  \\
         & \quad \Longleftrightarrow \quad \Tor_i^R\big(H^p(X),H^q(Z)\big) = 0 \textrm{ for every } p,q \in \bbZ \textrm{ and } i \ge 0.
\end{align*}

Finally notice that $H^q(Z) \cong H^q(C)$ as left $R$-modules for each $q\in\bbZ$. Indeed, consider the quasi-isomorphism $v\dd {_BV_C} \to {_BC_C}$ from the proof Lemma~\ref{lem:representing-bimod}. Clearly $H^q(C) \cong H^q(V)$ as left $R$-modules via $v$. If we apply $- \otimes_B V$ to the quasi-isomorphisms of dg $B$-modules $\sigma\dd B_B \to R_B$, we get a quasi-isomorphisms $V \cong B \otimes_B V \to R \otimes_B V = Z$ since ${_BV}$ is homotopically projective in $\Cpx{B\op}$, and the induced isomorphisms $H^q(V) \cong H^q(Z)$ are easily checked to be isomorphisms of left $R$-modules.
\end{proof}

We can make the statement of the last lemma even stronger, showing that only the zeroth cohomology is enough to determine the kernel of $- \Lotimes_B C$.

\begin{prop} \label{prop:Ho-is-enough}
Let $R$ be a $k$-algebra \st $\wgldim R\leq 1$ and $f\sigma\inv\dd R\to C$ be a homomorphism in $\Ho\Dga k$. 
Let $X\in \D(R)$; then the following are equivalent:
\begin{enumerate}
\item $X \Lotimes_B C=0$;
\item $X \Lotimes_R H^0(C)=0$;
\item $H^p(X) \Lotimes_R H^0(C)=0$ for every $p \in \bbZ$.
\end{enumerate}
\end{prop} 

\begin{proof}
Consider the class $\X = \Ker(-\Lotimes_B C) \subseteq \Der R$ as in Lemma~\ref{lem:kernel-of-localization}, and denote $S = H^0(C)$. Since $\sigma$ is a quasi-isomorphism, $f\sigma\inv$ induces a homomorphism $R \to S$ of $k$-algebras. In view of Proposition~\ref{prop:Kunneth} we ought to prove that
\[ \X = \big\{ X \in \Der R \mid \Tor_i^R\big(H^p(X),S\big) = 0 \textrm{ for all } p,q \in \bbZ \textrm{ and } i = 0,1 \big\}, \]

The inclusion $\subseteq$ is clear by Lemma~\ref{lem:kernel-of-localization}. For the other inclusion, suppose that $\Tor_i^R\big(H^p(X),S\big) = 0$ for all $p \in \bbZ$ and $i=0,1$. If $q\in \bbZ$ is arbitrary, we have
\[ H^p(X) \otimes_R H^q(C) = 0 \qquad \textrm{ for all } q \in \bbZ \]
since $H^q(C)$ is a left $S$-module and $H^p(X) \otimes_R S = 0$ by the assumption. Since we also assume that $\Tor_i^R(H^p(X),S) = 0$ for all $i \ge 1$, we have isomorphisms
\[ \Tor_n^R\big(H^p(X),H^q(C)\big) \cong \Tor_n^S\big(H^p(X) \otimes_R S,H^q(C)\big) \]
for all $p,q \in \bbZ$ and $n \ge 1$ by~\cite[Ch. VI, Proposition 4.1.1]{CE56} or~\cite[The Mapping Theorem]{Mitch73}. The latter term is zero thanks to our assumption that $H^p(X) \otimes_R S = 0$, showing that $X \in \X$ as required.
\end{proof}

Now we aim to state and prove the main result of the section. We remark that a result from~\cite{KrSt} says the same as the theorem below, but under a much more restrictive condition that $R$ is a one-sided hereditary ring.

\begin{thm} \label{thm:smashing-classif}
Let $R$ be a (possibly non-commutative) algebra of weak global dimension at most one over a commutative ring $k$. Then the assignment
\[
f \longmapsto \{ X \in \Der R \mid X \Lotimes_R S = 0 \}
\]
is a bijection between
\begin{enumerate}
\item equivalence classes of homological epimorphisms $f\dd R \to S$ originating at $R$, and
\item smashing localizing subcategories $\X \subseteq \Der R$.
\end{enumerate}
Moreover, the class $\X$ corresponding to a given $f$ consists precisely of the complexes $X \in \Der R$ \st $H^n(X) \otimes_R S = 0 = \Tor_1^R\big(H^n(X), S\big)$ for all $n \in \bbZ$.
\end{thm}

\begin{proof}
Suppose that $\X \subseteq \Der R$ is a smashing localizing class. Then by Remark~\ref{rem:recollement} and Proposition~\ref{prop:descr-dga}, viewing $R$ as a dg algebra concentrated in degree $0$, there is a homological epimorphism
\[
\xymatrix{
& B \ar[dl]_\sigma \ar[dr]^h \\
R && C,
}
\]
in $\Ho\Dga k$ \st $\sigma^*h_*\dd \Der C \to \Der R$ is fully faithful, the essential image of $\sigma^*h_*$ is $\X^\perp$, and $\X = \{ X \in \Der R \mid X \Lotimes_B C = 0 \}$ in $\Der R$. 
We can assume that $C$ is cofibrant in $\Dga k$ for the same reason as in the proof of Lemma~\ref{lem:kernel-of-localization}.
Let $S = H^0(C)$ and $f = H^0(h)H^0(\sigma)\inv\dd R \to S$ be the induced homomorphisms of ordinary $k$-algebras. By Proposition~\ref{prop:Ho-is-enough} we have
\begin{equation} \label{eqn:ker-S}
\X = \{ X \in \Der R \mid H^p(X) \Lotimes_R S = 0 \textrm{ for all } p \in \bbZ \}. 
\end{equation}

Now we consider the corresponding Bousfield localization functor $(L,\eta)$ from Remark~\ref{rem:recollement}, and the morphism $\eta_R\dd R \to L(R)$ in $\Der R$. As explained in Remark~\ref{rem:recollement}, this is none other than an $\X^\perp$-reflection of $R$ in $\Der R$ and that $L(R) \cong R \Lotimes_B C \Lotimes_B R$. Furthermore, as $\sigma\dd B \to R$ is a quasi-isomorphism, we obtain in $\Der R$ isomorphism
\[ L(R) \cong B \Lotimes_B C \Lotimes_B R \cong C \Lotimes_B R. \]

Observe further that $H^0(L(R)) \cong S_R$ as right $R$-modules. This is since we have isomorphisms of right $R$-modules $H^0(C \Lotimes_B R) \cong H^0(C \Lotimes_B B) \cong H^0(C)=S$, where the $R$-module structure of the second term is induced from the natural $H^0(B)$-module structure via the isomorphism of algebras $H^0(\sigma)\dd H^0(B) \to R$, and the $R$-module structure of the rightmost term is given by the homomorphism $f\dd R \to S$.

Since $L$ is a Bousfield localization, $L\eta = \sigma^*h_*(\eta\Lotimes_B C)$ is an isomorphism. As $\sigma^*h_*$ is fully faithful, $\eta\Lotimes_B C$ must be an isomorphism in $\Der C$. If we denote by $Q$ a mapping cone of $\eta_R$, we therefore have $Q \in \X$, and equality (\ref{eqn:ker-S}) together with Proposition~\ref{prop:Ho-is-enough} imply that $Q \Lotimes_R S = 0$ and $\eta_R \Lotimes_R S$ is a quasi-isomorphism in $\Cpx S$. Let $W_R$ be a homotopically projective resolution of $L(R)$ in $\Cpx R$, so that $L(R) \Lotimes_R S \cong W \otimes_R S$ and $H^0(W) \cong S_R$, and let $\eta'\colon R \to W$ be a lift of $\eta$. Then $\eta'\otimes_R S$ is a quasi-isomorphism. In particular $\eta'$ induces an isomorphism of $S$-modules $t\dd S \to H^0(W \otimes_R S)$ and $H^{-1}(W \otimes_R S) = 0$.

Applying K\"unneth's theorem to $W \otimes_R S$, we obtain a short exact sequence
\[ 0 \la H^0(W) \otimes_R S \mapr{i} H^0(W \otimes_R S) \la \Tor_1^R(H^1(W),S) \la 0 \]
of $S$-modules and the isomorphism $t$ factors as
\[ S \cong R \otimes_S S \overset{f \otimes 1_S}\la S \otimes_R S \overset{H^0(\eta')\otimes 1_S}\la H^0(W) \otimes_R S \overset{i}\la H^0(W \otimes_R S). \]
Thus, the morphism $i$, being a monomorphism and an epimorphism at the same time, is clearly an isomorphism, and so is the multiplication map $S \otimes_R S \to S$. Invoking K\"unneth's theorem once again, we also obtain a short exact sequence
\[ 0 \la H^{-1}(W) \otimes_R S \la H^{-1}(W \otimes_R S) \la \Tor_1^R(H^0(W),S) \la 0, \]
which tells us that $\Tor_1^R(S,S) = 0$. Hence $f\dd R \to S$ is a homological epimorphism.

Finally notice that if two homological epimorphisms $f\dd R \to S$ and $f'\dd R \to S'$ induce the same smashing localizing subcategory $\X$, then the essential images of $f_*\dd \Der S \to \Der R$ and $f'_*\dd \Der{S'} \to \Der R$ also coincide by Remark~\ref{rem:recollement}. Clearly also $\Img f_* \cap \ModR = \Img f'_* \cap \ModR$, which implies that the essential images of the restriction functors
\[ \Modr S \la \ModR \qquad \textrm{and} \qquad \Modr{S'} \la \ModR \]
are the same (see~\cite[Lemma 4.6]{AKL11}). Hence $f$ and $f'$ are equivalent homological epimorphisms by~\cite[Theorem 1.2]{GdlP87}.
\end{proof}

% ------------------------------------------------------------------------------
\section{A direct module theoretic approach}
\label{sec:loc-modules}

It is rather clear from the previous results that smashing localizations of $\Der R$ for an algebra of weak global dimension at most one can be mostly described using module categories rather than invoking derived categories. We will show here how this approach can be worked out.

Suppose $R$ is a $k$-algebra \st $\wgldim R \le 1$ and $\X \subseteq \Der R$ is a smashing localizing class. Then Theorem~\ref{thm:smashing-classif} and~\cite[Lemma 4.6]{AKL11} imply that given the corresponding TTF triple $(\X,\Y,\Z)$ in $\Der R$, there is a homological ring epimorphism $f\dd R \to S$ \st
\begin{align*}
\X &= \big\{ X \in \Der R \mid H^n(X) \otimes_R S = 0 = \Tor_1^R\big(H^n(X), S\big) \textrm{ for all } n \in \bbZ \big\},  \\
\Y &= \big\{ Y \in \Der R \mid H^n(Y) \in \Modr S \textrm{ for all } n \in \bbZ \big\}.
\end{align*}
Thus, both $\X$ and $\Y$ are determined by their intersections with $\ModR$, which we denote $\X_0$ and $\Y_0$, respectively. Adjusting the results from~\cite[\S2]{KrSt}, we will see that $(\X_0,\Y_0)$ is a what is called an Ext-orthogonal pair there.

\begin{thm} \label{thm:ext-ortho}
Let $R$ be a $k$-algebra of $\wgldim R \le 1$ and let $f\dd R \to S$ be a homological epimorphism. Denote
\[
\X_0 = \big\{ X' \in \ModR \mid X' \otimes_R S = 0 = \Tor_1^R(X', S) \big\}.
\]
Then, given any $M \in \ModR$, there is a $5$-term exact sequence
\[ \varepsilon_M\dd \quad 0 \to \Tor_1^R(M, S) \to X_M \to M \to M \otimes_R S \to X^M \to 0. \]
Moreover, the map $X_M \to M$ is an $\X_0$-coreflection and $M \to M \otimes_R S$ is a $(\Modr S)$-reflection. Therefore, $\varepsilon_M$ is unique up to a unique isomorphism and functorial in $M$, and the functor $M \mapsto \varepsilon_M$ commutes with direct limits.
\end{thm}

\begin{proof}
Let $\X = \{ X \in \Der R \mid H^p(X) \in \X_0 \textrm{ for all } p \in \bbZ\}$ be the smashing subcategory corresponding to $f$. In order to obtain $\varepsilon_M$, we simply consider the triangle
\[
X \la M \la M \Lotimes_R S \la \Sigma X
\]
as in \S\ref{subsec:Bousfield} with $X \in \X$ and apply the cohomology functor. Note that $M \to M \Lotimes_R S$ is a $\Y$-reflection functor by Remark~\ref{rem:recollement} and $X \to M$ is an $\X$-coreflection. It follows from Theorem~\ref{thm:smashing-classif} that $X_M, X^M \in \X_0$. Further, $\Ext_R^i(X',N) = 0$ for each $i \ge 0$, $X' \in \X_0$ and $N \in \Modr S$ because of the TTF triple from Remark~\ref{rem:recollement}. The fact that we have an $\X_0$-coreflection and $(\Modr S)$-reflection in $\varepsilon_M$ was proved in~\cite[Lemma 2.9]{KrSt}. Finally, since both $\X_0$ and $\Modr S$ are closed under direct limits in $\ModR$, the assignment $M \mapsto \varepsilon_M$ commutes with direct limits by the same argument as for~\cite[Lemma 5.3(2)]{KrSt}.
\end{proof}

Furthermore, the essential images of $\Modr S \to \ModR$ for homological epimorphisms $f\dd R \to S$ are simply determined by closure properties:

\begin{prop} \label{prop:hom-epi-semihered}
Let $R$ be a $k$-algebra of weak global dimension at most one. Then the assignment $f \mapsto \A = \Img f_*$ yields a bijective correspondence between
\begin{enumerate}
 \item equivalence classes of homological epimorphisms $f\dd R \to S$, and
 \item subcategories $\A \subseteq \ModR$ which are closed under limits, colimits and extensions.
\end{enumerate}
\end{prop}

\begin{proof}
If we start with $\A$ as in (2), there is an algebra epimorphism $f\dd R \to S$ \st $\Img f_* = \A$ and $f$ is unique up to equivalence; see~\cite[Theorem 1.2]{GdlP87}. If we denote by $DM$ the character module $\Hom_Z(M,\bbQ/\bbZ)$, then
\[ D\Tor_1^R(S,S) \cong \Ext^1_R(S,DS) \cong \Ext^1_S(S,DS) = 0, \]
where the second isomorphism follows from the fact that $\A$ is closed under extensions. So $f\dd R \to S$ is a homological epimorphism.

Suppose conversely that $f\dd R \to S$ is a homological epimorphism. Then $\Img f_*$ is closed under limits, colimits and extensions even without any restriction on $R$. Indeed, $\Img f_*$ is closed under limits and colimits since the forgetful functor $f_*\dd\Modr{S} \to \ModR$ is fully faithful and preserves limits and colimits.
Suppose that we have a short exact sequence $0 \to X \to Y \to Z \to 0$ in $\ModR$ \st $X,Z \in \Img f_*$. Then we have a commutative diagram with isomorphisms in the two marked columns
\[
\begin{CD}
        0     @>>>       X        @>>>       Y       @>>>       Z       @>>>     0     \\
@. @V{X \otimes_R f}V{\cong}V @V{Y \otimes_R f}VV @V{Z \otimes_R f}V{\cong}V           \\
\Tor_1^R(Z,S) @>>> X \otimes_R S  @>>> Y \otimes_R S @>>> Z \otimes_R S @>>>     0     \\
\end{CD}
\]
Since $\Tor_1^R(Z,S) \cong \Tor_1^S(Z,S) = 0$ by Proposition~\ref{prop:char-hom-epi}, $Y \otimes_R f$ is an isomorphism and $Y \in \Img f_*$.
\end{proof}

Finally, we mention a relation to universal localizations.

\begin{defn} \label{def:univ-loc} \cite[Ch. 4]{Scho85}
A ring homomorphism $f\dd R\to S$ is called a \emph{universal localization} if there exists a set $\mathfrak{S}$ of morphisms between finitely generated projective $R$-modules such that
\begin{enumerate}
\item $\sigma \otimes_R S$ is an isomorphism of $S$-modules for all $\sigma \in\mathfrak{S}$, and
\item every ring homomorphism $R\to S'$ such that $\sigma \otimes_R S'$ is an isomorphism of $S'$-modules for all $\sigma \in \mathfrak{S}$ factors uniquely through $f\dd R\to S$.
\end{enumerate} 
\end{defn}

Note that for any set $\mathfrak{S}$ of morphisms between finitely generated projective $R$-modules there is a corresponding universal localization by~\cite[Theorem 4.1]{Scho85}. If $R$ is a $k$-algebra, so is naturally $S$ via the composite structure map $k \to R \to S$. Moreover, the following hold:

\begin{lem} \label{lem:univ-loc-basic}
Let $R$ be a $k$-algebra and $f\dd R\to S$ be a universal localization at $\mathfrak{S}$. Then $f$ is a ring epimorphism, $\Tor^1_R(S,S)=0$, and the essential image of the forgetful functor $f_*\dd \Modr S \to \ModR$ consists precisely of the modules $M$ \st $\Hom_R(\sigma,M)$ is an isomorphism for each $\sigma \in \mathfrak{S}$.
\end{lem}

\begin{proof}
The universal localization is a ring epimorphism by the construction in~\cite[Theorem 4.1]{Scho85}, and $\Tor^1_R(S,S)=0$ by~\cite[Theorems 4.7 and 4.8]{Scho85}. In particular, $f_*$ is fully faithful by the discussion in \S\ref{subsec:hom-epi-rings}.

It remains to determine the essential image of $f_*$. Although this seems to be known to experts (compare with the proof of \cite[Theorem 4.7]{Scho85}), we give a full argument since we lack a good reference. To start with, for each $P,M \in \ModR$ and a $k$-module $N$, there is a homomorphism
\[ \varphi_{P,M,N}\dd P \otimes_R \Hom_k(M,N) \la \Hom_k(\Hom_R(P,M),N), \]
natural in all $P,M,N$, given by $\varphi_{P,M,N}(p\otimes u)(g) = u(g(p))$, where $p\in P$, $u \in \Hom_k(M,N)$ and $g \in \Hom_R(P,M)$. Since $\varphi_{P,M,N}$ is an isomorphism for $P = R_R$, it is also an isomorphism whenever $P$ is finitely generated projective over $R$. In particular, if $P$ is finitely generated projective and $M = N$, we have a natural isomorphism
\[ \varphi_{P,M,M}\dd P \otimes_R \End_k(M) \la \Hom_k(\Hom_R(P,M),M). \]

Suppose now that $M \in \ModR$. If $\Hom_R(\sigma,M)$ is an isomorphism for all $\sigma\in\mathfrak{S}$, so are $\Hom_k(\Hom_R(\sigma,M),M)$ and also $\sigma \otimes_R \End_k(M)$. Hence the structure morphism $R \to \End_k(M), r \mapsto -\cdot r$ extends to $S \to \End_k(M)$ by the universal property of $f$, giving $M$ an $S$-module structure.

If conversely $M \in \Img f_*$, then $\Hom_R(\sigma,M)$ is an isomorphism is an isomorphism \iff $\Hom_S(\sigma\otimes_RS,M)$ (we use that $f_*\dd \ModR \to \Modr{S}$ has a left adjoint $-\otimes_RS$). However, $\sigma\otimes_RS$ is an isomorphism for each $\sigma\in\mathfrak{S}$ by the definition of $f$.
%then $\sigma \otimes_R \End_k(M)$ is an isomorphism and so is $\Hom_k(\Hom_R(\sigma,M),M)$. However, then also $\Hom_k(\Hom_R(\sigma,M),N)$ is an isomorphism whenever $N$ is a summand of a finite direct sum of copies of $M$ in $\Modr k$. Since both ends of $\Hom_R(\sigma,M)$ are of this form, $\Hom_R(\sigma,M)$ is an isomorphism by the Yoneda lemma.
\end{proof}

The final proposition will be useful for future reference, especially in connection with the Telescope Conjecture studied in Section~\ref{sec:tc}.

\begin{prop} \label{prop:univ-loc-wd1}
Let $R$ be an algebra of weak global dimension at most one. If $f\dd R \to S$ is a universal localization, it is a homological epimorphism and it corresponds to a compactly generated localization in the correspondence of Theorem~\ref{thm:smashing-classif}.

If, moreover, $R$ is right semihereditary (that is, every finitely generated submodule of a right projective module is again projective), then the bijection from Theorem~\ref{thm:smashing-classif} restricts to a bijection between
\begin{enumerate}
\item equivalence classes of universal localizations $f\dd R \to S$, and
\item compactly generated localizing subcategories $\X \subseteq \Der R$.
\end{enumerate}
\end{prop}

\begin{proof}
Since $\Tor^R_n(S,S) = 0$ for $n \ge 2$ by assumption, $f$ is a homological epimorphism by Lemma~\ref{lem:univ-loc-basic}. Since a complex $Y$ is in the essential image of $f_*\dd \Der S \to \Der R$ \iff its cohomology is an $S$-module (see~\cite[Lemma 4.6]{AKL11}), we have
\begin{align*}
\Img f_* &=
\{ Y \in \Der R \mid \Hom_R(\sigma,H^n(Y)) \textrm{ is an iso for all } \sigma \in \mathfrak{S} \textrm{ and } n \in \mathbb{Z} \} \\ 
&= \{ Y \in \Der R \mid \RHom_R(\sigma,Y) \textrm{ is an isomorphism for all } \sigma \in \mathfrak{S} \}.
\end{align*}
Thus, the smashing localization of $\Der R$ corresponding to $f$ in the sense of Theorem~\ref{thm:smashing-classif} is generated by $\mathfrak{S}$, when we view its elements as 2-term perfect complexes.

Suppose now that $R$ is right semihereditary. Then $R$ is right coherent and the category $\modR$ of finitely presented right $R$-modules is hereditary abelian. Moreover, if $C$ is a perfect complex in $\Der R$, then $C=\bigoplus_{n\in \bbZ}H^n(C)[-n]$ and for every $n\in \bbZ$, $H^n(C)$ is a finitely presented module (see for instance \cite[Section 2.5]{Kel07} or \cite[Section 1.6]{Kr07}). Thus, any compactly generated localization of $\Der R$ is generated by a set of finitely presented $R$-modules and this yields precisely the same result as universally inverting (any choice of) projective resolutions of these modules.
\end{proof}

% ------------------------------------------------------------------------------
\section{The classification for valuation domains}
\label{sec:classif-vd}

Now we are in a position to classify all smashing localizations (equivalently: homological epimorphisms) for valuation domains $R$, i.e.\ commutative domains whose ideals are totally ordered by inclusion. This will also reveal the amount of information which we need to know about $R$ in order to reconstruct the lattice of smashing localizations (see~\cite{Kr05-telescope,BaFa11}): Knowing just the Zariski spectrum as a topological space is not enough (see Example~\ref{expl:Puiseux} below), we also need to know which of the prime ideals are idempotent.

For properties of ideals of valuation domains we refer to \cite[Chapter II]{FS01}. In the sequel we will use without further mention that the kernel of a homological ring epimorphism $\phi\colon R\to S$ is an idempotent ideal (see Lemma~\ref{lem:kernel-idempotent}) and that an idempotent ideal of a valuation domain is a prime ideal. Note also that if $\p \subseteq \q$ are prime ideals, then $\p$ is canonically an $R_\q$-module, so that $\p_\q = \p$ and $(R/\p)_\q = R_\q/\p$.

Before starting our work on the classification, we state a useful lemma which also explains why valuation domains are a natural starting point.

\begin{lem} \label{lem:locally-VD}
Let $R$ be a commutative ring. Then $\wgldim R\leq 1$ \iff $R_\p$ is a valuation domain for every prime ideal $\p$ of $R$.
\end{lem}

\begin{proof}
We refer to~\cite[Corollary 4.2.6]{Glaz89}.
\end{proof}

% --------------------------------------
\subsection{From a homological epimorphism to a collection of intervals}
\label{subsec:hom-epi-to-intervals}

We will first show that a homological epimorphism $f\dd R \to S$ naturally induces a collection of disjoint intervals of $\Spec R$ satisfying certain conditions. Idempotent ideals will play an important role and, if $S$ is semilocal, this will readily yield an explicit description of $S$.

\begin{nota} \label{nota:iSpec}
We shall denote the collection of all idempotent ideals of $R$ by $\iSpec R$ and view $(\iSpec R, \subseteq)$ as a totally ordered subset of the Zariski spectrum $(\Spec R, \subseteq)$.
\end{nota}

The following are easy properties of the idempotent spectrum.

\begin{lem} \label{lem:lots-of-idempotents}
Let $R$ be a valuation domain and $\clS \subseteq \Spec R$ be a set. Then:

\begin{enumerate}
\item If $\clS$ has no maximal element with respect to inclusion, then $\bigcup \clS$ is an idempotent ideal. 
\item Any subset $\clS \subseteq \iSpec R$ has a supremum and an infimum in $\iSpec R$.
\end{enumerate}
\end{lem}

\begin{proof}
For the first part, one directly checks that $\p = \bigcup \clS$ is a prime ideal. By \cite[Lemma II.4.3(iv) and property (d), p. 69]{FS01} a prime ideal in a valuation domain is either idempotent or it is a principal ideal of the localization $R_\p$. Now, $\p \subseteq R_\p$ is a maximal ideal of $R_\p$ and it cannot be principal in $R_\p$ since $\p = \bigcup \clS$. 
The second statement is a direct consequence of the first one.
\end{proof}

As an initial step in our classification we shall describe flat ring epimorphisms.

\begin{prop} \label{prop:flat-epi-VD}
Let $R$ be a valuation domain and let $f\dd R\to S$ be a flat ring epimorphism. Then $f$ is injective and there is a prime ideal $\p$ of $R$ such that $f$ is equivalent to the localization morphism $R \to R_\p$.
\end{prop}

\begin{proof} 
The kernel of $f$ must vanish, since $S$ is a flat $R$-module, hence torsion free. Localizing $f$ at the zero ideal of $R$, we obtain the injective ring epimorphism $f \otimes_R Q\dd Q \to S \otimes_R Q$ where $Q$ is the quotient field of $R$. Thus $f \otimes_R Q$ is an isomorphism (see~\cite[Corollary IV.1.3]{Laz69} or apply Proposition~\ref{prop:hom-epi-semihered} to $\Modr Q$) and, since $S$ is flat, we have a ring monomorphism $R\otimes_R S\to Q\otimes_R S$ and thus, up to isomorphism, a chain of ring embeddings $R \subseteq S \subseteq Q$. Now consider the set $\mathfrak{S} = \{ r \in R \setminus \{0\} \mid r\inv \in S \}$. One easily checks that $\mathfrak{S}$ is a saturated (closed under divisors) multiplicative set in $R$, that $\p = R \setminus \mathfrak{S}$ is a prime ideal, and that $S = R_\p$.
\end{proof} 

In order to understand general homological epimorphisms, we establish a connection between the maximal ideals of a homological factor of $R$ and the promised collections of intervals in the poset $(\Spec R,\subseteq)$.

\begin{prop} \label{prop:hom-epi-vd}
Let $R$ be a valuation domain, $0 \ne f\dd R \to S$ be a homological epimorphism, and denote $\ifr = \Ker f$. Then the following hold:
\begin{enumerate}
\item There exists a prime ideal $\p \in \Spec R$ with $\ifr \subseteq \p$ and a surjective homological epimorphism $g\dd S \to R_\p/\ifr$ \st the composition $gf\dd R \to R_\p/\ifr$ is the canonical morphism.
Moreover, there is a unique maximal ideal $\n$ of $S$ \st $g\dd S \to R_\p/\ifr$ is equivalent to the localization of $S$ at $\n$.

\item If $\n$ is a maximal ideal of $S$, then the localization morphism $S \to S_\n$ is surjective and the composition
\[ R \mapr{f} S \mapr{\mathrm{can}} S_\n \]
is a homological epimorphism equivalent to $g\dd R \to R_\q/\jfr$, where $\jfr \subseteq R$ is an idempotent ideal, $\q = f\inv(\n)$ and $\jfr \subseteq \q$.

\item If $\n' \ne \n$ is another maximal ideal and $\jfr' \subseteq \q'$ are the corresponding primes in $R$ with $\jfr'$ idempotent, then the intervals
\[ [\jfr, \q] \qquad \textrm{and} \qquad [\jfr', \q'] \]
in $(\Spec R, \subseteq)$ are disjoint. In particular, we have for all $n \ge 0$ that
\[ \Tor_n^R( R_\q/\jfr, R_{\q'}/\jfr' ) = 0 \]
\end{enumerate}
\end{prop}

\begin{proof}
Note that $S$ is a commutative ring by Proposition~\ref{prop:silver-lazard}(1) and that $\wgldim S\leq 1$ by Lemma~\ref{lem:kernel-idempotent}. In particular, every localization of $S$ at a prime ideal is a valuation domain by Lemma~\ref{lem:locally-VD}.
   
By Lemma~\ref{lem:kernel-idempotent}, $f$ induces a homological ring epimorphism $R/\ifr\to S$, where $R/\ifr$ is a valuation domain, since $\ifr$ is a prime ideal. Thus, without loss of generality we may assume that $\ifr=0$. 

(1) Viewing $S$ as an $R$-module via $f$, we shall consider its torsion submodule $t(S)=\{x\in S\mid \exists\ 0\neq r\in R \textrm{ \st} rx=0\}$. Clearly, $J=t(S)$ is an ideal of $S$ and $S/J$ is a torsion free, hence flat $R$-module. The composition $R \to S\to S/J$ is then a flat epimorphism and, by Proposition~\ref{prop:flat-epi-VD}, there is a prime ideal $\p$ of $R$ such $R \to S/J$ is equivalent to the canonical morphism $R \to R_\p$.

In particular, $S/J$ is a local ring and we shall consider the unique maximal ideal $\n$ of $S$ which contains $J$. Clearly $\n$ is mapped to $\p$ under the surjection $g\dd S \to R_\p$ and consequently $\p = f\inv(\n)$. Note also that the composition $h\dd R \overset{f}\to S \to S_\n$ is injective. Indeed, $\Ker h$ consists of all elements $r \in R$ whose annihilator in $S$ is not contained in $\n$, but the $R$-torsion part $J = t(S)$ is contained in $\n$, so $\Ker h = 0$. Therefore, $h\dd R \to S_\n$ is a flat epimorphism since $S_\n$ is a domain, and the combination of Proposition~\ref{prop:flat-epi-VD} with the equality $\p = f\inv(\n)$ tells us that $h$ and $R \to R_\p$ are equivalent.

(2) Let now $\n$ be an arbitrary maximal ideal of $S$ and let $\psi_{\n}\colon S\to S_{\n}$ be the localization map. Consider the homological ring epimorphism $\psi_{\n}f\dd R\to S_{\n}$ and let $\jfr$ be the kernel of $\psi_{\n}f$. Then $\jfr$ is an idempotent prime ideal of $R$, and $S_{\n}$ is a flat $R/\jfr$-module since $S_{\n}$ is a domain. By Proposition~\ref{prop:flat-epi-VD} there is a prime ideal $\q$ of $R$ containing $\jfr$ \st $S_{\n}\cong R_{\q}/\jfr$ and we necessarily have $\q = f\inv(\n)$.

Moreover, since $\jfr S$ vanishes under the localization $\psi_\n\dd S \to S_\n$ by the very definition of $\jfr$, $\psi_\n$ canonically factors as $S \to S/\jfr S \to (S/\jfr S)_{\n/\jfr S} \cong S_\n$. In particular, the $R/\jfr$-torsion part $t'(S/\jfr S)$ of $S/\jfr S$ is contained in $\n/\jfr S$ and, by part (1), the epimorphism $S/\jfr S \to (S/\jfr S)_{\n/\jfr S}$ is equivalent to $S/\jfr S \to (S/\jfr S)/t'(S/\jfr S)$. In particular, $S \to S_\n$ is surjective.

(3) Suppose that we have two distinct maximal ideals $\n,\n'$ in $S$ and the corresponding pairs $\jfr \subseteq \q$ and $\jfr' \subseteq \q'$ of ideals in $R$ as in (2). Assume without loss of generality that $\jfr \subseteq \jfr'$. Since the localization map $\psi_\n\dd S \to S_\n$ is surjective by (2) and in particular $\n$ is the unique maximal ideal containing $J = \Ker \psi_\n$, we have $\Tor_n^S(S/J, S_{\n'})=0$ for all $n \geq 0$. Indeed, there exists $x \in J \setminus \n'$ and the multiplication by $x$ must act on the Tor groups as zero and as an isomorphism at the same time. Thus,
\[ 0 = \Tor_n^S(S/J, S_{\n'}) = \Tor_n^S(S_\n, S_{\n'}) \cong \Tor_n^R(R_\q/\jfr, R_{\q'}/\jfr'), \]
since $f$ is a homological ring epimorphism (see Proposition~\ref{prop:char-hom-epi}(2)), and $S_\n \cong R_\q/\jfr$ and $S_{\n'} \cong R_{\q'}/\jfr'$.

It remains to prove that the intervals $[\jfr, \q]$ and $[\jfr', \q']$ of $(\Spec R,\subseteq)$ are disjoint. This is easy now since if $\jfr' \subseteq \q$, then we would have $R_\q/\jfr \otimes_R R_{\q'}/\jfr' \cong R_{\q \cap \q'}/\jfr' \ne 0$, a contradiction.
\end{proof}

\begin{rem} \label{rem:semilocal-VD}
From the latter proposition the structure of homological epimorphisms $f\dd R \to S$ with $S$ semilocal is rather clear. In such a case the finitely many maximal ideals $\n_i \subseteq S$ give us finitely many pairwise disjoint intervals $[\jfr_i,\p_i]$ in $(\Spec R, \subseteq)$ with all $\jfr_i$ idempotent. Proposition~\ref{prop:hom-epi-vd} also provides us with a homological epimorphism
\[ h\dd S \la\prod_i R_{\p_i}/\jfr_i. \]
Moreover, $h_{\n_i}$ is an isomorphism for every maximal ideal $\n_i \subseteq S$, so that $h$ itself is an isomorphism.
\end{rem}

The non-semilocal case is more difficult. We shall focus on the problem which collections of intervals can occur in the conclusion of Proposition~\ref{prop:hom-epi-vd}.

\begin{defn} \label{def:intervals}
Let $R$ be a valuation domain. An \emph{admissible interval} $[\ifr,\p]$ is an interval in $(\Spec R, \subseteq)$ \st $\ifr^2 = \ifr \subseteq \p$. The set of all admissible intervals will be denoted by $\Inter R$. We equip $\Inter R$ with a partial order: $[\ifr,\p] < [\ifr',\p']$ if $\p \subsetneqq \ifr'$ as ideals.

If $f\dd R \to S$ is a homological epimorphism, we denote by $\I(f)$ the collection of all admissible intervals $[\jfr,\q]$ which occur as in Proposition~\ref{prop:hom-epi-vd}(2).
\end{defn}

Thus, our task is to analyze the properties of $\I(f)$. First of all, it is easy to relate $\I(f)$ to the spectrum of $S$ as a poset.

\begin{lem} \label{lem:spectrum-epi}
Let $R$ be a valuation domain and $f\dd R \to S$ a homological epimorphism. Then the canonical map
\[ \Spec S \la \Spec R, \qquad \n \mapsto f\inv(\n) \]
restricts to a poset isomorphism between $(\Spec S,\subseteq)$ and the coproduct (=~disjoint union)
\[ \coprod_{[\jfr,\q] \in \I(f)} [\jfr,\q], \]
where $[\jfr,\q]$ are viewed as subchains of $(\Spec R,\subseteq)$.
\end{lem}

\begin{proof}
By Definition~\ref{def:intervals}, there is a bijection between maximal ideals of $S$ and elements of $\I(f)$. The rest follows from the fact that the primes below a maximal ideal $\n \subseteq S$ correspond to the primes in the valuation domain~$S_\n$.
\end{proof}

%However, The Zariski topology on $\Spec S$ is in general more complicated than just the coproduct topology. An simple reason for this is that $\Spec S$ must be quasi-compact. We shall discuss this more in detail in \S\ref{subsec:lots-of-idemp}, but a hint is given by the following proposition.
Having this description at hand, the coming proposition encodes a crucial necessary condition on possible infinite collections of intervals coming from homological epimorphisms.

\begin{prop} \label{prop:no-gaps-hom-epi}
Let $R$ be a valuation domain and $f\dd R \to S$ be a homological epimorphism.

\begin{enumerate}
\item If $\clS = \{ [\jfr_\ell,\q_\ell] \mid \ell \in \Lambda \} \subseteq \I(f)$ is a non-empty subset with no minimal element, then $\I(f)$ contains an element of the form $[\jfr,\bigcap_{\ell\in\Lambda} \q_\ell]$.

\item If $\clS = \{ [\jfr_\ell,\q_\ell] \mid \ell \in \Lambda \} \subseteq \I(f)$ is a non-empty subset with no maximal element, then $\I(f)$ contains an element of the form $[\bigcup_{\ell\in\Lambda} \jfr_\ell,\q]$.
\end{enumerate}
\end{prop}

\begin{proof}
(1) Denote $\p = \bigcap_{\ell \in \Lambda} \q_\ell = \bigcap_{\ell \in \Lambda} \jfr_\ell$ and, appealing to Lemma~\ref{lem:lots-of-idempotents}, let $\ifr = \inf\{ \jfr_\ell \mid \ell \in \Lambda \}$ be the infimum taken in $(\iSpec R,\subseteq)$.
One directly checks that $\p$ is a prime ideal. Although it may happen that $\ifr \subsetneqq \p$ (see Example~\ref{expl:non-coherent} below), we at least know that any idempotent prime ideal contained in $\p$ has to be contained in $\ifr$. That is, there is no idempotent ideal $\p'$ \st $\ifr \subsetneqq \p' \subseteq \p$. We finish the proof of (1) in two steps.

Step (a): We claim that $\I(f)$ must contain an element of the form $[\jfr,\q]$ \st $\jfr \subseteq \ifr \subseteq \q \subseteq \p$.

To see that, let us denote for each $[\jfr_\ell,\q_\ell]$ the corresponding maximal ideal of $S$ by  $\n_\ell$. Then $\jfr_\ell = \Ker (R \to S_{\n_\ell})$ contains $\ifr$ for each $\ell \in \Lambda$. In particular, each $\psi_{\n_\ell}\dd S \to S_{\n_\ell}$ canonically factors through the projection $p\dd S \to S/\ifr S$ and we infer that the kernel of the composition $pf\dd R \to S/\ifr S$ is contained in $\p = \bigcap_{\ell \in \Lambda} \jfr_\ell$. Since $pf$ is a homological epimorphism and $\Ker pf$ is idempotent by Lemma~\ref{lem:kernel-idempotent}, we must have $\Ker pf \subseteq \ifr$ and then clearly $\Ker pf = \ifr$. Let $\m \supseteq \ifr S$ be the unique maximal ideal of $S$ \st $\m/\ifr S$ fits Proposition~\ref{prop:hom-epi-vd}(1) when applied to $pf\dd R \to S/\ifr S$. If $[\jfr,\q] \in \I(f)$ is the interval corresponding to $\m$, then $\jfr = \Ker (R \to S_\m) \subseteq \ifr$ and $\ifr \subseteq f\inv(\m) = \q$ by construction. This proves the claim and reduces our task to showing that $\q = \p$.

Step (b): We claim that $\q = \p$.

To see that, suppose by way of contradiction that $\q \subsetneqq \p$ and consider an element $x \in \p \setminus \q$. This means that $x^n \ne \q$ for each $n \ge 1$ and, since $R$ is a valuation domain, we have $\q \subsetneqq x^nR \subseteq \p $.

We shall denote $y = f(x)$ and by $u\dd S \to S[1/y]$ the corresponding localization. Suppose that $\n \subseteq S$ is a maximal ideal and $[\jfr',\q'] \in \I(f)$ the corresponding interval in $\Spec R$. Then $S[1/y]_\n \cong (R_{\q'}/\jfr')[1/x]$ canonically. Hence $S_\n \to S[1/y]_\n$ is either a zero map or an isomorphism depending on whether $x \in \jfr'$ or not (in such a case $x \notin \q'$ by the choice of $x$).
In particular $S \to S[1/y]$ is surjective since localizations at all maximal ideals jointly reflect exactness.

Thus, $\Spec S[1/y]$ can be identified with a quasi-compact clopen set $V \subseteq \Spec S$ whose complement $\Spec R \setminus V$ is also quasi-compact. As in the noetherian case, we find an idempotent $e \in S$ such that $S \to S[1/y]$ is equivalent to $S \to S[1/e] \cong S/(1-e)S$ as a ring epimorphism. Indeed, there exists $t \in S$ \st $1/y = t$ in $S[1/y]$, which forces the existence of $n \ge 0$ \st $y^n = t y^{n+1}$ in $S$. In particular $y^n = t^n y^{2n}$ and $e = t^n y^n$ is the idempotent which we are looking for.

Thus we have a homological epimorphism $g\dd R \to S/eS$ and $\Ker g = f\inv(eS)$ is an idempotent ideal in $R$ by Lemma~\ref{lem:kernel-idempotent}. As $y^n = e y^n \in eS$, we have $x^n \in \Ker g$. In particular $\q \subsetneqq \Ker g$. On the other hand $\Ker g \subseteq \p$ since $e = t^n y^n$ vanishes under $R \to S_{\n_\ell} \cong R_{\q_\ell}/\jfr_\ell$ for each $\ell \in \Lambda$. To summarize, we have constructed an idempotent ideal $\Ker g \subseteq R$ \st $\q \subsetneqq \Ker g \subseteq \p$, in contradiction to our assumption that there are no idempotent ideals in that interval.

Hence $\q = \p$ and the proof of step (b) as well as statement (1) of the proposition are finished.

(2) Let $\p = \bigcup_{\ell \in \Lambda} \jfr_\ell$. It suffices to prove that $S \otimes_R k(\p) \ne 0$ for the residue field $k(\p) = R_\p/\p$, for then $S_\n \otimes k(\p) \ne 0$ for some maximal ideal $\n$ of $S$ and thus $R_\q/\jfr \otimes k(\p) \ne 0$ for some $[\jfr,\q] \in \I(f)$. Since the intervals in $\I(f)$ are disjoint this implies that $\p = \jfr$.

Now note that $S \otimes_R R_\p/\jfr_\ell \ne 0$ for each $\ell \in \Lambda$ since we can always find $\ell' \in \Lambda$ and $\n_{\ell'} \in \Spec S$ \st $[\jfr_\ell,\q_\ell] < [\jfr_{\ell'},\q_{\ell'}]$, $\q_{\ell'} = f\inv(\n_{\ell'})$ and
\[ S_{\n_{\ell'}} \otimes_R R_\p/\jfr_\ell \cong R_{\q_{\ell'}}/\jfr_{\ell'} \otimes_R R_\p/\jfr_\ell \cong R_{\q_{\ell'}}/\jfr_{\ell'} \ne 0. \]
Observe further that $S \otimes_R k(\p)$ can be expressed as a direct limit of $(S \otimes_R R_\p/\jfr_\ell \mid \ell \in \Lambda)$, where the maps in the direct system are surjective. If $1\otimes_R 1 \in S \otimes_R k(\p)$ were zero, standard properties of direct limits would imply that also $0 = 1\otimes_R 1 \in S \otimes_R R_\p/\jfr_\ell$ for some $\ell \in \Lambda$, a contradiction.
\end{proof}

% --------------------------------------
\subsection{Mazet presentations and abundance of idempotents}
\label{subsec:lots-of-idemp}

So far we have mostly used the homological properties of $f\dd R \to S$. Now we are going to employ the fact that $f$ is a ring epimorphism. The following concept will facilitate our discussion.

\begin{defn} \label{def:dominion}
Let $R,S$ be arbitrary (non-commutative) rings and $f\dd R \to S$ be a ring homomorphism. The \emph{dominion} of $f$ is the collection of all elements $s \in S$ \st for any pair $g_1,g_2\dd S \to T$ of ring homomorphisms with $g_1f= g_2f$ we have also $g_1(s) = g_2(s)$.
\end{defn}

In connection to homotopy theory, dominions of ring homomorphisms $\bbZ \to S$ were also studied in~\cite{BK72,BK72-corr}.
The following zig-zag criterion for the elements in the dominion was originally studied by Mazet~\cite{Maz68}. It was stated in the present form by Isbell~\cite{Isb69}, who combined Mazet's results with those in~\cite{Sil67}.

\begin{prop}[Isbell-Mazet-Silver] \label{prop:dominion}
Let $f\dd R \to S$ be a ring homomorphism and $s \in S$. Then the following are equivalent:
\begin{enumerate}
\item $s$ is in the dominion of $f$;
\item There exist natural numbers $m,n \ge 1$ and matrices $X \in M_{1\times m}(S)$, $Y \in M_{m\times n}(R)$ and $Z \in M_{n\times 1}(S)$ \st $s = X \cdot f(Y) \cdot Z$ (as $1\times 1$-matrices) and $X \cdot f(Y), f(Y) \cdot Z$ are (row and column, respectively) matrices over $\Img f$.
\end{enumerate}
\end{prop}

\begin{proof}
The implication $(2) \Rightarrow (1)$ is straightforward. The argument for the converse is sketched in~\cite[Theorem 1.1]{Isb69}. If $s \in S$ belongs to the dominion, the proof of~\cite[Proposition 1.1]{Sil67} shows that $sm = ms$ for any $S$-$S$-bimodule $M$ and any element $m \in M$. Applying this to $M = S \otimes_R S$ and $m = 1 \otimes_R 1$ implies that $s \otimes_R 1 = 1 \otimes_R s$. As explicitly computed in~\cite{Maz68} or~\cite[\S1.4]{GdlP87}, the latter has as a consequence the existence of matrices $X,Y,Z$ as in~(2).
\end{proof}

\begin{cor} \label{cor:dominion}
Given $f\dd R \to S$, the dominion is the largest subring $S'\subseteq S$ \st $f\dd R \to S'$ is a ring epimorphism.
\end{cor}

This leads to the following definition.

\begin{defn} \label{def:mazets-pres}
Let $f\dd R \to S$ be a ring homomorphism and $s \in S$ be in the dominion. Then a triple $(P,Y,Q) \in M_{1\times n}(R) \times M_{m\times n}(R) \times M_{m\times 1}(R)$ is called a \emph{Mazet presentation} of $s$ over~$R$ if there exist matrices $X \in M_{1\times m}(S)$ and $Z \in M_{n\times 1}(S)$
 \st
\[
s = X \cdot f(Y) \cdot Z, \qquad f(P) = X \cdot f(Y), \quad \textrm{and} \quad f(Q) =  f(Y) \cdot Z.
\]
\end{defn}

Note that the image of $s$ under any ring homomorphism $g\dd S \to T$ (including $g = id_S$) is fully determined by $(P,Y,Q)$. In fact, only $P$ and $Q$ suffice, but it will be more convenient for us to work with $Y$ as well. For valuation domains, the situation simplifies as follows.

\begin{lem} \label{lem:mazet-VD}
Let $R$ be a valuation domain and $f\dd R \to S$ be a ring epimorphism. Then every $s \in S$ has a Mazet presentation $(P,Y,Q)$ \st $Y$ is a diagonal square matrix.
\end{lem}

\begin{proof}
Let $(P,Y,Q)$ be an arbitrary Mazet presentation for $s \in S$. First we can turn $Y = (y_{ij})$ into a Smith normal form (that is, $y_{ij} = 0$ unless $i=j$ and $R \supseteq y_{11}R \supseteq y_{22}R \supseteq \cdots$) by applying equivalent row and column operations to $Y$ and changing $P$ and $Q$ correspondingly. Indeed, the same proof as for discrete valuation domains applies and this is again closely related to the fact that, if we consider $Y$ as a presentation matrix of an $R$-module $N$, then $N$ is a direct sum of cyclically presented modules by~\cite[Theorem I.7.9]{FS01}.

Second, if $Y$ is a diagonal $m \times n$ matrix, we can crop it to a square matrix of size $\min(m,n)$ and truncate $P$ and $Q$ correspondingly. As we have left out only zero entries, this will still be a presentation for $s$.
\end{proof}

Consider now a homological epimorphism $f\dd R \to S$. The following is an easy consequence of the results in~\S\ref{subsec:hom-epi-to-intervals}.

\begin{lem} \label{lem:embed-S}
Let $R$ be a valuation domain and $f\dd R \to S$ be a homological epimorphism. Then we can identify $S$ with a subring of $\prod_{[\jfr,\q] \in \I(f)} R_\q/\jfr$ and $f$ with the canonical map obtained from $R \to \prod_{[\jfr,\q] \in \I(f)} R_\q/\jfr$ by restriction.
\end{lem}

\begin{proof}
Clearly the homomorphism $S \to \prod_{\n \in \Max S} S_\n$ is injective. The rest is easily deduced from Proposition~\ref{prop:hom-epi-vd} since we can canonically identify each $S_\n$ with $R_\q/\jfr$ for some $[\jfr,\q] \in \I(f)$.
\end{proof}

Now we establish the key fact: The components of any fixed element $s \in S$ in $\prod_{[\jfr,\q] \in \I(f)} R_\q/\jfr$ can come only from finitely many elements in $R$.

\begin{prop} \label{prop:loc-const}
Let $R$ be a valuation domain and $f\dd R \to S$ be a homological epimorphism with $S \subseteq \prod_{[\jfr,\q] \in \I(f)} R_\q/\jfr$ as in the above lemma. Fix an element $s = (s_{[\jfr,\q]})_{[\jfr,\q] \in \I(f)}$. Then there exist an integer $k \ge 1$, and for each $1 \le j \le k$ an interval $[\p_j,\p'_j]$ in $\Spec R$ and an element $r_j \in R_{\p'_j}/\p_j$ such that:
\begin{enumerate}
\item Every $[\jfr,\q] \in \I(f)$ is contained in $[\p_j,\p'_j]$ for some $1 \le j \le k$. That is, we have $\p_j \subseteq \jfr \subseteq \q \subseteq \p'_j$.
\item Whenever $1 \le j \le k$ and $[\jfr,\q] \in \I(f)$ are such that $[\jfr,\q]$ is contained in $[\p_j,\p'_j]$, then $s_{[\jfr,\q]}$ is the image of $r_j$ under the canonical map $R_{\p'_j}/\p_j \to R_\q/\jfr$.
\end{enumerate}
\end{prop}

\begin{proof}
By Proposition~\ref{prop:no-gaps-hom-epi}(2) and Zorn's lemma, $\I(f)$ possesses a (unique) interval which is maximal with respect to the order on $\Inter R$. Let us denote this interval by $[\ifr',\n]$. Consider now a Mazet presentation $(P,Y,Q)$ of $s \in S$ where $Y$ is a square diagonal matrix, and consider all principal ideals of $R$ generated by the entries in $P,Y,Q$ which are contained in $\ifr'$. Ordering these ideals by inclusion and removing duplicities results in a finite list
\[ I_1 \subsetneqq I_2 \subsetneqq \dots \subsetneqq I_\ell \]
of principal ideals of $R$. In order to facilitate the discussion, we also put $I_0 = 0$ and take $I_{\ell+1} = \ifr'$.

Now consider an integer $j$ \st $0 \le j \le \ell$, let $\p_j = \sqrt I_j$ and let $\p'_j$ be the maximal prime ideal \st $\p'_j \subseteq I_{j+1}$. Since a non-zero non-maximal prime ideal cannot be principal, we have $I_j \subsetneqq \p_j$ unless $j=0$ and $\p'_j \subsetneqq I_{j+1}$ unless $j=\ell$. In particular, the images of $P,Y,Q$ under the canonical map $g_j\dd R \to R_{\p_j'}/\p_j$ have either zeros or units in all entries. Now there are two possibilities: either there is an element $r_j \in R_{\p'_j}/\p_j$ with the Mazet presentation $(P,Y,Q)$, or there is none. This depends only on  whether the systems of linear equations $X \cdot g_j(Y) = g_j(P)$ and $g_j(Y) \cdot Z = g_j(Q)$ have solutions $X,Z$ over $R_{\p_j'}/\p_j$. In the first case such an element $r_j$ is unique and whenever $[\jfr,\q] \in \I(f)$ is contained in $[\p_j,\p_j']$, then $s_{[\jfr,\q]}$ must be the image of $r_j$. In the second case we claim that there cannot exist any $[\jfr,\q] \in \I(f)$ which is contained in $[\p_j,\p_j']$. Indeed, the systems of equations $X \cdot g_j(Y) = g_j(P)$ and $g_j(Y) \cdot Z = g_j(Q)$ have a very easy form since $g_j(Y)$ is diagonal and every entry of $g_j(Y),g_j(P),g_j(Q)$ is either a unit or zero. If there existed $[\jfr,\q] \in \I(f)$ inside $[\p_j,\p_j']$, units would stay units and zeros would stay zeros under the homomorphism $R_{\p'_j}/\p_j \to R_\q/\jfr$, so there could not be any element with the Mazet presentation $(P,Y,Q)$ in $R_\q/\jfr$ either. But clearly $s_{[\jfr,\q]}$ is such an element, a contradiction. This proves the claim.

Let us put all the intervals $[\p_j,\p_j']$ for which there exists $r_j$ as above on our list. It only can happen that some $[\jfr,\q] \in \I(f)$ is not covered by any such $[\p_j,\p_j']$ if either we have $\jfr \subseteq I_i \subseteq \q$ for some $0 \le i \le \ell$, or if $[\jfr,\q] = [\ifr',\n]$. In either case, we simply add $[\jfr,\q]$ and $s_{[\jfr,q]}$ to our list, resulting in at most finitely many additional intervals. By the very construction, we have obtained a collection of intervals and elements as in the statement.
\end{proof}

\begin{rem} \label{rem:irredundant}
By possibly removing finitely many intervals from the collection obtained by Proposition~\ref{prop:loc-const}, we may without loss of generality assume that the collection is irredundant, i.e.
\begin{enumerate}
\item no $[\p_j,\p'_j]$ is contained in $[\p_i,\p'_i]$ for any $i \ne j$, and
\item each $[\p_j,\p'_j]$ contains an interval from $\I(f)$.
\end{enumerate}
If we order the intervals such that $\p_1 \subseteq \p_2 \subseteq \dots \subseteq \p_k$, the irredundancy implies that $\p'_j \subseteq \p'_{j+1}$ for each $1 \le i < k$.
\end{rem}

Another important fact is that a much stronger reduction of the number of intervals is possible. As it turns out, we will be able to glue together any pair of overlapping intervals thanks to the following instance of the sheaf axiom (for a scheme-theoretic interpretation see~\cite[Theorem 3.3]{Sch12}).

\begin{lem} \label{lem:gluing}
Let $R$ be a valuation domain, let $k \ge 1$, and suppose that we are given for each $1 \le j \le k$ an interval $[\p_j,\p'_j]$ in $\Spec R$ and an element $r_j \in R_{\p'_j}/\p_j$. Suppose further that
\begin{enumerate}
\item $\p_j \subseteq \p_{j+1} \subseteq \p'_j \subseteq \p'_{j+1}$, and
\item the images of $r_j$ and $r_{j+1}$ under the canonical maps coincide in $R_{\p'_j}/\p_{j+1}$
\end{enumerate}
for each $1 \le j < k$. Then there exists a unique element $r \in R_{\p'_k}/\p_1$ \st the image of $r$ under $R_{\p'_k}/\p_1 \to R_{\p'_j}/\p_j$ equals $r_j$ for each $1 \le j \le k$.
\end{lem}

\begin{proof}
There is nothing to prove for $k=1$ and the case $k=2$ just amounts to the straightforward checking that the square with canonical maps
\[
\begin{CD}
R_{\p'_2}/\p_1   @>>>   R_{\p'_1}/\p_1   \\
    @VVV                     @VVV        \\
R_{\p'_2}/\p_2   @>>>   R_{\p'_1}/\p_2
\end{CD}
\]
is a pull-back of rings. Note that there we can without loss of generality assume that $\p_1=0$ and $\p'_2$ is the maximal ideal.

We proceed by induction for $k > 2$. By inductive hypothesis, there is a unique element $r' \in R_{\p'_{k-1}}/\p_1$ \st the image of $r'$ under the canonical map $R_{\p'_{k-1}}/\p_1 \to (R_{\p'_j}/\p_j)$ equals $r_j$ for all $1 \le j < k$. Furthermore, the images of $r'$ and $r_k$ in $R_{\p'_{k-1}}/\p_k$ coincide by assumption (2) applied to $j = k-1$. Thus, we can glue $r'$ and $r_k$ to a unique element $r \in R_{\p'_k}/\p_1$ using the argument for $k=2$.
\end{proof}

As a consequence, we obtain the following dichotomy.

\begin{prop} \label{prop:dichotomy}
Let $R$ be a valuation domain, let $f\dd R \to S$ be a homological epimorphism with $S \ne 0$, and let $\I(f)$ be the collection of intervals as in Definition~\ref{def:intervals}. Then
\begin{enumerate}
\item either $\I(f)$ contains a single element,
\item or there are intervals $[\jfr,\q] < [\jfr',\q']$ in $\I(f)$ with no other interval of $\I(f)$ between them.
\end{enumerate}
In particular, either $S$ is local or it has a non-trivial idempotent element.
\end{prop}

\begin{proof}
Suppose that (2) does not hold, or equivalently that the order on $\I(f)$ is dense. Suppose further that $s = (s_{[\jfr,\q]})_{[\jfr,\q] \in \I(f)} \in S$ and $[\p_j,\p'_j]$ and $r_j \in R_{\p'_j}/\p_j$ is a corresponding collection of intervals and elements as in Proposition~\ref{prop:loc-const}, which is irredundant and ordered as in Remark~\ref{rem:irredundant}.

We first claim that then $\p_{j+1} \subseteq \p'_j$ for each $1 \le j < k$. Indeed, suppose to the contrary that for instance $\p'_1 \subsetneqq \p_2$. Then using Proposition~\ref{prop:no-gaps-hom-epi} and Zorn's lemma we can find
\begin{itemize}
\item a maximal element $[\jfr',\q']$ among those elements of $\I(f)$ which are contained in $[\p_1,\p'_1]$ and
\item a minimal element $[\jfr'',\q'']$ among those elements of $\I(f)$ which are contained in $[\p_2,\p'_2]$.
\end{itemize}
Since each element of $\I(f)$ is contained in some interval $[\p_j,\p'_j]$, there cannot exist any element of $\I(f)$ between $[\jfr',\q']$ and $[\jfr'',\q'']$, contradicting that $\I(f)$ is densely ordered. This establishes the claim.

Since also Lemma~\ref{lem:gluing}(2) is satisfied for any collection of intervals and elements coming from the proof of Proposition~\ref{prop:loc-const} (all $r_j$ have the same Mazet presentation over $R$), there is an element $r \in R_{\p_k}/\p_1$ such that the image of $r$ under $R_{\p_k}/\p_1 \to R_\q/\jfr$ equals $s_{[\jfr,\q]}$ for each $[\jfr,\q] \in \I(f)$.

Let us rephrase what we have just shown. Thanks to Proposition~\ref{prop:no-gaps-hom-epi}, $\I(f)$ has a unique minimal element $[\ifr,\p]$ and a unique maximal element $[\ifr',\n]$. If $\I(f)$ is densely ordered, we have shown that for every $s \in S$ there exists $r \in R_\n/\ifr$ \st $s$ is the image of $r$ under the morphism $R_\n/\ifr \to S$ induced by $f$. Put yet in other words, if $\I(f)$ is densely ordered, then $R_\n/\ifr \to S$ is surjective, $S$ is necessarily local, and $\I(f)$ has a single element by the very definition. This proves the dichotomy between (1) and (2) in the statement.

For the second part, suppose that $S$ is not local, fix some $[\jfr,\q] < [\jfr',\q']$ in $\I(f)$ with no other interval between them and fix $x \in \jfr' \setminus \q$. Then $S \to S[1/y]$ is surjective for $y = f(x)$ since $S_{\n} \to S_\n[1/y]$ is either zero or an isomorphism for every $\n \in \Max S$. Now the same argument as in step~(b) of the proof of Proposition~\ref{prop:no-gaps-hom-epi}(1) provides us with a non-trivial idempotent $e \in S$.
\end{proof}

\begin{cor} \label{cor:disconnected}
Given a homological epimorphism $f\dd R \to S$ where $R$ is a valuation domain, and given any $[\jfr_0,\q_0] < [\jfr_1,\q_1]$ in $\I(f)$, there are intervals $[\jfr,\q],[\jfr',\q']$ in $\I(f)$ \st
\[ [\jfr_0,\q_0] \le [\jfr,\q] < [\jfr',\q'] \le [\jfr_1,\q_1] \]
and there is no other interval in $\I(f)$ between $[\jfr,\q]$ and $[\jfr',\q']$.
\end{cor}

\begin{proof}
This follows by applying Proposition~\ref{prop:dichotomy} to the composition $R \overset{f}\to S \to S_{\q_1}/\jfr_0S$.
\end{proof}

In the non-local case, $S$ is formally similar to a von Neumann regular ring in that the Zariski topology on $\Max S$ is totally disconnected. If $S$ is semihereditary, this similarity can be formalized by noting that the localization of $S$ at the set of all regular elements is von Neumann regular by~\cite[Corollary 4.2.19]{Glaz89}.
Note also that in our situation, the regular elements of $S$ are precisely those $s = (s_{[\jfr,\q]})$ for which each component $s_{[\jfr,\q]}$ is non-zero. Beware, however, that $S$ might \emph{not} be semihereditary:

\begin{expl} \label{expl:non-coherent}
Suppose that $R$ is a valuation domain with a countable descending chain $\ifr_1 \supseteq \ifr_2 \supseteq \cdots$ of idempotent ideals \st the intersection $\q = \bigcap \ifr_i$ is not idempotent. Such an example can be constructed by means of~\cite[Theorem II.3.8]{FS01}, where the value group $(G,\le)$ is taken as follows: We put $H = \bbQ^{(\omega)}$ (a countable direct sum of copies of $\bbQ$) with the antilexicographic order and $G = \bbZ \times H$ with the components lexicographically ordered.

Denote now $R_j = R_{\ifr_{j+1}} \times k(\ifr_j) \times k(\ifr_{j-1}) \times \cdots \times k(\ifr_1)$, where $k(\ifr_j)$ is the residue field of $\ifr_j$, and consider the chain of obvious ring homomorphisms
\[ R \to R_1 \to R_2 \to \cdots \to S = \varinjlim_j R_j. \]
One can check (see the results in \S\ref{subsec:intervals-to-hom-epi}) that $R \to S$ is a homological epimorphism and $\I(f) = \{[0,\q]\} \cap \{[\ifr_j,\ifr_j] \mid j \ge 1 \}$. But there are elements $s = (s_{[\jfr,\q]}) \in S$ \st $s_{[0,q]} \ne 0$ and $s_{[\ifr_j,\ifr_j]} = 0$ for all $j \ge 1$. Then the ideal $sS$ is not finitely presented, hence $S$ is not semihereditary. One can also check that every regular element of $S$ is a unit, so the localization at the set of regular elements is not von Neumann regular.
\end{expl}

% --------------------------------------
\subsection{From intervals to a homological epimorphism}
\label{subsec:intervals-to-hom-epi}

Now we finish the classification of homological epimorphisms starting at a valuation domain. In particular, given a suitable collection $\I \subseteq \Inter R$ (Definition~\ref{def:intervals}) we construct the corresponding homological epimorphism $f_{(\I)}\dd R \to R_{(\I)}$.

\begin{constr} \label{constr:non-semilocal-case}
Suppose that $R$ is a valuation domain and $(\I,\le)$ is a non-empty subchain of $(\Inter R,\le)$ satisfying the conditions implied by Proposition~\ref{prop:no-gaps-hom-epi} and Corollary~\ref{cor:disconnected}. That is, we require:
\begin{enumerate}
\item[(C1)] If $\clS = \{ [\jfr_\ell,\q_\ell] \mid \ell \in \Lambda \}$ is a non-empty subset of $\I$ with no minimal element, then $\I$ contains an element of the form $[\jfr,\bigcap_{\ell\in\Lambda} \q_\ell]$.
\item[(C2)] If $\clS = \{ [\jfr_\ell,\q_\ell] \mid \ell \in \Lambda \}$ is a non-empty subset of $\I$ with no maximal element, then $\I$ contains an element of the form $[\bigcup_{\ell\in\Lambda} \jfr_\ell,\q]$.
\item[(C3)] Given any pair $[\jfr_0,\q_0] < [\jfr_1,\q_1]$ in $\I$, then there are elements $[\jfr,\q],[\jfr',\q']$ in $\I$ \st
\[ [\jfr_0,\q_0] \le [\jfr,\q] < [\jfr',\q'] \le [\jfr_1,\q_1] \]
and there is no other interval in $\I$ between $[\jfr,\q]$ and $[\jfr',\q']$.
\end{enumerate}
Denote by $[\ifr,\p]$ the unique minimal element of $\I$ and by $[\ifr',\n]$ the unique maximal element. Further denote by $R_\I$ the ring product $\prod_{[\jfr,\q] \in \I} R_\q/\jfr$ and by $g_\I\dd R_\n/\ifr \to R_\I$ the canonical ring homomorphism. Clearly $g_\I$ is an embedding.

Consider now a partition of $\I$ into a finite disjoint union $\I = \I_0 \cup \cdots \cup \I_n$ of chains in $\Inter R$ which satisfies two simple conditions:
\begin{enumerate}
\item[(a)] Each $\I_\ell$, $0 \le \ell \le n$, is a subchain of $\I$ and has a minimal element $[\ifr_\ell,\p_\ell]$ and a maximal element $[\ifr'_\ell,\m_\ell]$.
\item[(b)] If $j < \ell$, then $[\jfr,\q] < [\jfr',\q']$ for each $[\jfr,\q] \in \I_j$ and $[\jfr',\q'] \in \I_\ell$.
\end{enumerate}
In other words, we have subdivided $\I$ into finitely many intervals which enjoy properties (1)--(3) as $\I$ does itself.

Using this notation, we define a map
\[ g_{(\I_0, \dots, \I_n)}\dd \prod_{\ell = 0}^n R_{\n_\ell}/\ifr_\ell \la R_\I \]
as the composition of the product of the maps
\[ g_{\I_\ell} \dd R_{\n_\ell}/\ifr_\ell \la R_{\I_\ell} \]
with the obvious isomorphism $\prod_{\ell=0}^n R_{\I_\ell} \cong R_\I$. Again $g_{(\I_0, \dots, \I_n)}$ is an embedding.

Another easy observation reveals that the images of $g_{(\I_0, \dots, \I_n)}$, where $(\I_0, \dots, \I_n)$ varies over all partitions of $\I$ satisfying conditions (a) and (b) above, form a direct system of subrings of $R_\I$. We denote by $R_{(\I)}$ the direct union of all these images and by
\[ f_{(\I)}\dd R \la R_{(\I)} \]
the ring homomorphism induced by the composition $R \to R_\n/\ifr \overset{g_\I}\to R_\I$.
\end{constr}

The highlight of the section is the following theorem, which together with Theorem~\ref{thm:smashing-classif} classifies smashing localizations of $\Der R$.

\begin{thm} \label{thm:hom-epi-vd}
Let $R$ be a valuation domain. Then there is a bijection between:
\begin{enumerate}
\item[(i)] Subchains $\I$ of $\Inter R$ (cf.\ Definition~\ref{def:intervals}) which satisfy conditions (C1)--(C3) from Construction~\ref{constr:non-semilocal-case}.
\item[(ii)] Equivalence classes of homological epimorphisms $f\dd R \to S$.
\end{enumerate}

The bijection is given by assigning to a non-empty $\I$ from (i) the ring homomorphism $f_{(\I)}\dd R \to R_{(\I)}$ from Construction~\ref{constr:non-semilocal-case}. We assign $R \to 0$ to $\I = \emptyset$.
The converse is given by sending $f\dd R \to S$ to $\I = \I(f)$ (see Definition~\ref{def:intervals}).
\end{thm}

\begin{proof}
Suppose we have $\I \ne \emptyset$ as in $(1)$ and consider $f_{(\I)}\dd R \to R_{(\I)}$. Since $f_{(\I)}$ is a direct limit of homological epimorphisms of the form
\begin{equation} \label{eqn:finite-tensor}
R \la \prod_{\ell = 0}^n R_{\n_\ell}/\ifr_\ell, \qquad
\ifr_0 \subseteq \n_0 \subsetneqq \ifr_1 \subseteq \n_1 \subsetneqq \cdots \subsetneqq \n_n \subseteq \ifr_n,
\end{equation}
and since the Tor functors commute with direct limits, it follows that $f_{(\I)}$ is a homological epimorphism.

Suppose now that $\I$ is a set of admissible intervals as in (i) and let $\I' = \I(f_{(\I)})$. We claim that $\I' = \I$. To this end, let $[\ifr,\p] \in \I'$. As then $R_{(\I)} \otimes_R R_\p/\ifr \ne 0$ by the very definition of $\I'$, we deduce that there is an interval $[\jfr_0,\q_0] \in \I$ which overlaps $[\ifr,\p]$. Indeed, otherwise we could take $\I_0 = \{ [\jfr,\q] \in \I \mid [\jfr,\q] < [\ifr,\p] \}$, $\I_1 = \{ [\jfr,\q] \in \I \mid [\jfr,\q] > [\ifr,\p] \}$ and $R_{\n_0}/\ifr_0 \times R_{\n_1}/\ifr_1$ as in Construction~\ref{constr:non-semilocal-case} (using conditions (C1) and (C2) on $\I$), but then $(R_{\n_0}/\ifr_0 \times R_{\n_1}/\ifr_1) \otimes_R R_\p/\ifr = 0$, so $R_{(\I)} \otimes_R R_\p/\ifr = 0$, a contradiction. Note further that since $R_\p/\ifr$ is isomorphic to a localization of $R_{(\I)}$ at a maximal ideal as an $R$-algebra, the obvious morphism $R_\p/\ifr \to R_{(\I)} \otimes_R R_\p/\ifr$ is an isomorphism. This implies that $[\jfr_0,\q_0]$ contains $[\ifr,\p]$. Indeed, otherwise we would encounter one of the following two cases:

\begin{enumerate}
\item There is another interval in $\I$ overlapping $[\ifr,\p]$. Then $R_{(\I)} \otimes_R R_\p/\ifr$ would contain a nontrivial idempotent by Construction~\ref{constr:non-semilocal-case}, using condition (C3). This is a contradiction to $R_{(\I)} \otimes_R R_\p/\ifr \cong R_\p/\ifr$ being local.

\item The interval $[\jfr_0,\q_0]$ is the only interval overlapping $[\ifr,\p]$ and either $\jfr_0 \subseteq \ifr \subseteq \q_0 \subsetneqq \p$ or $\ifr \subsetneqq \jfr_0 \subseteq \p \subseteq \q_0$ or $\ifr \subsetneqq \jfr_0 \subseteq \q_0 \subsetneqq \p$. In the first case we can take $\I_0 = \{ [\jfr,\q] \in \I \mid [\jfr,\q] \le [\jfr_0,\q_0] \}$, $\I_1 = \{ [\jfr,\q] \in \I \mid [\jfr,\q] > [\ifr,\p] \}$ and $R_{\q_0}/\ifr_0 \times R_{\n_1}/\ifr_1$ as in Construction~\ref{constr:non-semilocal-case} (using condition (C1) to show that $\I_1$ has a minimum). Then $(R_{\q_0}/\ifr_0 \times R_{\n_1}/\ifr_1) \otimes_R R_\p/\ifr \cong R_{\q_0}/\ifr$ and also $R_{(\I)} \otimes_R R_\p/\ifr \cong R_{\q_0}/\ifr \not\cong R_\p/\ifr$, a contradiction. The other two cases lead to similar contradictions.
\end{enumerate}
To summarize, we know so far that each $[\ifr,\p] \in \I'$ is contained in a unique $[\jfr_0,\q_0] \in \I$.

Suppose conversely that we start with $[\jfr_0,\q_0] \in \I$. By the construction of $R_{(\I)}$ we have $R_{(\I)} \otimes_R R_{\q_0}/\jfr_0 \cong R_{\q_0}/\jfr_0$ as $R$-algebras since all terms in the defining direct system have this property. Thus $(R_{(\I)}/\jfr_0 R_{(\I)})_{\q_0}$ is local and there is a unique prime ideal $\n \in \Spec R_{(\I)}$ containing $\jfr_0 R_{(\I)}$ such that the localization of $R_{(\I)}/\jfr_0 R_{(\I)}$ at $\n$ is isomorphic to $R_{\q_0}/\jfr_0$ as $R$-algebra. Thus, invoking Lemma~\ref{lem:spectrum-epi} for $f = f_{(\I)}$, there is a unique interval $[\ifr,\p] \in \I' = \I(f_{(\I)})$ which contains $[\jfr_0,\q_0]$. This establishes the claim $\I' = \I$.

We have shown so far that the assignments $\I \mapsto f_{(\I)}$ and $f \to \I(f)$ are well defined maps between the appropriate sets (recall Proposition~\ref{prop:no-gaps-hom-epi} and Corollary~\ref{cor:disconnected}), and that the composition $\I \mapsto f_{(\I)} \mapsto \I(f_{(\I)})$ is the identity on chains of admissible intervals. In order to prove the theorem, it suffices to show that $\I \mapsto f_{(\I)}$ is a surjective assignment.

Thus suppose that we have $f\dd R \to S$ with $\I = \I(f)$. It is easy to check that $f$ uniquely factors through any morphism $R \to \prod_{\ell = 0}^n R_{\n_\ell}/\ifr_\ell$ in the direct system for  $f_{(\I)}\dd R \to R_{(\I)}$ in Construction~\ref{constr:non-semilocal-case}. One can see that  for instance by Proposition~\ref{prop:hom-epi-semihered}, using the fact that the canonical homomorphism $S \cong S \otimes_R \prod_{\ell = 0}^n R_{\n_\ell}/\ifr_\ell$ is bijective, which can be checked by localizing at maximal ideals of $S$. In particular we have a canonical morphism $g\dd R_{(\I)} \to S$. Since $\I(f) = \I(f_{(\I)})$, the map $(R_{(\I)})_\n \to S_\n$ is an isomorphism for each $\n \in \Max R_{(\I)}$. Thus $g$ is an isomorphism and we are done.
\end{proof}

\begin{expl} \label{expl:Puiseux}
Let $R$ be the ring which is called $A$ in~\cite[\S2]{Kel94-smash}. The same ring can be obtained by invoking~\cite[Theorem II.3.8]{FS01} for the totally ordered group $(\bbZ[1/\ell],+)$. Thus, $R$ is a valuation domain, $\Spec R = \{0,\m\}$ by~\cite[Proposition II.3.4]{FS01}, and $\m^2 = \m$. Let $Q$ be the quotient field of $R$ and $k = R/\m$ be the residue field. Our theorem says that we have precisely $5$ distinct homological epimorphisms starting at $R$: $R \to 0$, $R \to Q$, $R \to R$, $R \to k$ and $R \to Q \times k$, and only the first three are flat.
\end{expl}

% ------------------------------------------------------------------------------
\section{Flat epimorphisms}
\label{sec:flat}

Now we will turn back to general commutative rings of weak global dimension $\le 1$. Our aim is to understand flat ring epimorphisms in this case. As it turns out, they precisely correspond to compactly generated Bousfield localizations, but the proof seems rather non-trivial.

Let $R$ be commutative non-degenerate ring (i.e.\ $R\neq 0$), let $\wgldim R \le 1$ and let $f\dd R\to S$ be a flat ring epimorphism. Recall that by Lemma~\ref{lem:locally-VD}, $R_{\p}$ is a valuation domain for every prime ideal $\p$ of $R$. Thus, for every maximal ideal $\m$ of $R$, $f \otimes_R R_{\m}\dd R_{\m}\to S_{\m}$ is a flat epimorphism of the valuation domain $R_{\m}$. Hence by Proposition~\ref{prop:flat-epi-VD} there is a prime ideal $s(\m)$ of $R$ such that $s(\m)\subseteq \m$ and $f \otimes_R R_{\m}$ is equivalent to the localization $R_{\m}\to R_{s(\m)}$. Since localizations at all maximal ideals jointly reflect exactness, the map $f$ is necessarily injective.

We record the above notation for future reference.

\begin{nota} \label{nota:flat-epi-wgldim1}
If $0 \neq R$ is a ring with $\wgldim R \le 1$, $f\dd R\to S$ is a flat ring epimorphism and $\m \in \Max R$, we denote by $s(\m)$ the unique prime ideal of $R$ such that $s(\m)\subseteq \m$ and $f \otimes_R R_{\m}$ is equivalent to $R_{\m}\to R_{s(\m)}$.
\end{nota}

\begin{lem}\label{lem:non-hom-support}
In the situation of Notation~\ref{nota:flat-epi-wgldim1}, we have:
\[\{\q\in \Spec R\mid \q S=S\}= \{\q \in \Spec R\mid s(\m)\subsetneqq \q\subseteq  \m, \forall\m\in \Max R, \m\supseteq \q\}.\]
\end{lem}

\begin{proof}
For every prime ideal $\p\in \Spec R$ we have an exact sequence
\[0\la \p\otimes_R S\la R\otimes_RS\la R/\p\otimes_RS\la 0,\]
thus we may identify $\p\otimes_RS$ with the $S$-ideal $\p S$.

Let $\q\in \Spec R$ be such that $\q S=S$. Then, for every maximal ideal $\m$ of $R$, $\q S_{\m}=S_{\m}\cong R_{s(\m)}$, hence $\q\nsubseteq s(\m)$. Thus, if $\m \supseteq \q$ we must have $s(\m)\subsetneqq \q$ since all primes below $\m$ are totally ordered by inclusion.

Conversely,  let $\q\in \Spec R$ be such that $s(\m)\subsetneqq \q$ for every maximal ideal $\m$ of $R$ containing $\q$. Assume, by way of contradiction that $\q S\subsetneqq S$. Then there is a maximal ideal $\m$ of $R$ such that $\q S_{\m}\subsetneqq S_{\m}$. Thus, $\q R_{s(\m)}\subsetneqq R_{s(\m)}$, giving $\q\subseteq s(\m)\subseteq \m$, a contradiction.
\end{proof}

We aim to prove that every flat epimorphism $f\dd R \to S$ as above is given by a compactly generated localization of $\Der R$. The key role is played by Thomason's localization theory~\cite{Th97} which classifies compactly generated localizations purely in terms of $\Spec R$ as a topological space. Let us recall the fundamentals.

\begin{defn} \label{def:Thomason}
Let $R$ be a commutative ring. For $X \in \Der R$ we define its \emph{cohomological support} as
\[ \Supp X = \{ \p \in \Spec R \mid X \otimes_R R_\p \ne 0 \}. \]
For a class of complexes $\X$, we define $\Supp\X = \bigcup_{X \in \X} \Supp X$.

A subset $Z \subseteq \Spec R$ is a \emph{Thomason set} if it can be expressed as a union $Z = \bigcup Z_i$ with each $Z_i$ Zariski closed and \st $\Spec R \setminus Z_i$ is quasi-compact. In other words, we have $Z_i = \{ \p \in \Spec R \mid \p \supseteq I_i \}$ for a finitely generated ideal $I_i \subseteq R$.
\end{defn}

\begin{prop} \label{prop:Thomason-loc}
Let $R$ be a commutative ring. Then there is a bijection between
\begin{enumerate}
\item Thomason sets $Z \subseteq \Spec R$, and
\item compactly generated localizing subcategories $\X \subseteq \Der R$.
\end{enumerate}
given by the assignment $\X \mapsto Z = \Supp\X$.
\end{prop}

\begin{proof}
We combine two results from the literature. Firstly~\cite[Theorem 3.15]{Th97} provides us with a similar bijection between Thomason sets and thick subcategories (i.e.\ triangulated and closed under summands) of the category of perfect complexes. In particular, if $Z$ is a Thomason set and $\C$ is any set of perfect complexes \st $Z = \Supp\C$, then the smallest thick subcategory of $\Der R$ containing $\C$ is
\[ \C' = \{ X \in \Der R \mid X \emph{ compact and } \Supp X \subseteq Z \}. \]
Secondly, \cite[Theorem 2.1]{Nee92-loc} establishes a bijection between thick subcategories of the category of perfect complexes and compactly generated Bousfield localizations of $\Der R$ (cp.\ Definition~\ref{def:comp-gen-loc}). Starting with $\C$ as before, the localizing class $\X$ will be the smallest localizing class containing $\C$. Clearly $\Supp\C \subseteq \Supp\X$ and one easily sees also $\Supp\C \supseteq \Supp\X$ as closing $\C$ under coproducts, mapping cones and (de)suspensions cannot enlarge the support.
\end{proof}

\begin{rem} \label{rem:Thomason-loc}
Note that in the above correspondence, we only have proved $\X \subseteq \{ X \in \Der R \mid \Supp X \subseteq Z \}$, where $Z = \Supp\X$. We do not know whether these classes are equal in general, although they are in various cases. If $R$ is commutative noetherian, the equality essentially follows from~\cite[Lemma 3.6]{Nee92}. If $Z$ is Zariski closed with quasi-compact complement, the equality holds by~\cite[Theorem 2.2.4]{KoPi13}. As we will show below, the equality also holds whenever $\wgldim R \le 1$.
\end{rem}

To this end, we will need an auxiliary lemma which tells us how the support theory behaves \wrt localization.

\begin{lem} \label{lem:Thomason-base-chg}
Let $R$ be a commutative ring \st $\wgldim R \le 1$, $\mathfrak{S} \subseteq R$ be a multiplicative subset, and $\ell\dd R \to R_\mathfrak{S}$ be the localization morphism. Suppose that $\C \subseteq \Der R$ is a set of perfect complexes and $f\dd R \to S$ is a homological epimorphism corresponding to the Bousfield localization compactly generated by $\C$ (see Theorem~\ref{thm:smashing-classif} and~\S\ref{subsec:TTF-triples}).
Then $f \otimes_R R_\mathfrak{S}\dd R_\mathfrak{S} \to S_\mathfrak{S}$ corresponds to the Bousfield localization of $\Der{R_\mathfrak{S}}$ compactly generated by $\C_\mathfrak{S} = \{ C \Lotimes_R R_\mathfrak{S} \mid C \in \C \}$.
\end{lem}

\begin{proof}
Let $\Spec\ell\dd \Spec R_\mathfrak{S} \to \Spec R$ be the morphism between the spectra induced by $\ell$. 
Clearly, $\Supp \C_\mathfrak{S} = (\Spec\ell)\inv(\Supp\C)$ and the conclusion then follows from~\cite[Proposition 2.6]{Ste14}.

In more pedestrian terms, consider the set $\C'$ of perfect complexes of the form $R \stackrel{r}\to R$ which are concentrated in degrees $-1$ and $0$ and with $r \in \mathfrak{S}$. Then $\C' \otimes_R S = \{ S \stackrel{f(r)}\to S \} \subseteq \Der S$ compactly generates the localization of $\Der S$ whose corresponding homological epimorphism is the ordinary localization $S \to S_{f(\mathfrak{S})}$ \wrt the multiplicative subset $f(\mathfrak{S}) \subseteq S$. Thus the composition $R \to S \to S_{f(\mathfrak{S})}$ corresponds to the localization of $\Der R$ compactly generated by $\C \cup \C'$. The same composition can be also expressed as $R \to R_\mathfrak{S} \to S_{f(\mathfrak{S})}$, from which we see that $R_\mathfrak{S} \to S_{f(\mathfrak{S})}$ corresponds to the localization of $\Der{R_\mathfrak{S}}$ generated by $\C_\mathfrak{S} = \C \otimes_R R_\mathfrak{S}$.
\end{proof}

Now we can give the promised description of the class of acyclic objects for rings of weak global dimension at most $1$.

\begin{prop} \label{prop:Thomason-loc-wgldim1}
Suppose that $R$ is a commutative ring of $\wgldim R \le 1$ and $\X$ is a compactly generated localizing class in $\Der R$. Let $Z \subseteq \Spec R$ be the corresponding Thomason set and let $f\dd R \to S$ be the induced homological epimorphism. Then $f$ is a flat epimorphism and we have
\[
Z = \{ \q \in \Spec R \mid \q S = S \}
\quad \textrm{and} \quad
\X = \{ X \in \Der R \mid \Supp X \subseteq Z \}. \]
\end{prop}

\begin{proof}
Suppose first that $R$ is a valuation domain, hence semihereditary. If $\X$ is generated as a localizing class by compact objects, it is by Proposition~\ref{prop:univ-loc-wd1} and its proof generated by a set of finitely presented $R$-modules (viewed as complexes concentrated in degree $0$). Since every finitely presented module over a valuation domain is a direct sum of modules of the form $R/rR$ for some $r \in R$ (see~\cite[Theorem I.7.9]{FS01}), it follows that every compactly generated Bousfield localization is generated by a set of $2$-term perfect complexes of the form $R \stackrel{r}\to R$. As in the proof of Lemma~\ref{lem:Thomason-base-chg}, such a localization corresponds in terms of homological epimorphisms to an ordinary localization with respect to a multiplicative set. For a valuation domain, such a localization must be of the form $R \to R_\p$ for $\p \in \Spec R$; see the proof of Proposition~\ref{prop:flat-epi-VD}. Hence we have $Z = \Supp\X = \{ \q \in \Spec R \mid r \in \q \textrm{ for some } r \in R \setminus \p \} = \{ \q \in \Spec R \mid \q \supsetneqq \p \} = \{ \q \in \Spec R \mid \q R_\p = R_\p \}$. Moreover, $\X$ is as required by Theorem~\ref{thm:smashing-classif}.

Let now $R$ be general and $\C$ be a set of perfect complexes generating~$\X$. Then the morphism $f \otimes_R R_\m\dd R_\m \to S_\m$ for each $\m \in \Max R$ is by Lemma~\ref{lem:Thomason-base-chg} equivalent to $R_\m \to R_{s(\m)}$ as in Notation~\ref{nota:flat-epi-wgldim1}. In particular each $S_\m$ is flat over $R_\m$ and so $S$ is flat over $R$. Further, $\q \in Z$ \iff $\q \in \Supp\C$ \iff $\q \in \Supp\C_\m$ for each $\m \in \Max R$ \st $\m \supseteq \q$. Applying Lemma~\ref{lem:Thomason-base-chg}, we have  $Z = \{ \q \in \Spec R \mid \q R_{s(\m)} = R_{s(\m)} \; \forall \m \in \Max R, \m \supseteq \q \} = \{ \q \in \Spec R \mid \q S = S \}$.

Finally, by Theorem~\ref{thm:smashing-classif}, Lemma~\ref{lem:non-hom-support} and the above discussion we have
\begin{multline*}
\X = \big\{ X \in \Der R \mid H^n(X) \otimes_R S = 0 \textrm{ for all } n \in \bbZ \big\} = \\
   = \big\{ X \in \Der R \mid H^n(X) \otimes_R R_{s(\m)} = 0 \textrm{ for all } n \in \bbZ \textrm{ and } \m \in \Max R \big\} = \\
   = \{ X \in \Der R \mid \Supp X \subseteq Z \}.
\qedhere
\end{multline*}
\end{proof}

As a consequence, we get the characterization of homological epimorphisms coming from compactly generated Bousfield localizations of $\Der R$.

\begin{thm} \label{thm:flat-epi}
Let $R$ be a commutative ring of weak global dimension at most $1$. Then the correspondence from Theorem~\ref{thm:smashing-classif} restricts to the bijection between
\begin{enumerate}
\item equivalence classes of flat epimorphisms $f\dd R \to S$ originating at $R$, and
\item compactly generated localizing subcategories $\X \subseteq \Der R$.
\end{enumerate}
If, moreover, $R$ is semihereditary, then $f\dd R \to S$ is flat \iff $f$ is a universal localization.
\end{thm}

\begin{rem} \label{rem:semihereditary}
Note that semiheredity is a strictly stronger assumption than $\wgldim R \le 1$, see Example~\ref{expl:non-coherent} or~\cite[Example 3.1.2]{Glaz05}.
\end{rem}

\begin{proof}[Proof of Theorem~\ref{thm:flat-epi}]
If $f$ corresponds to a compactly generated Bousfield localizations, it is flat by Proposition~\ref{prop:Thomason-loc-wgldim1}.

Suppose conversely that $f\dd R \to S$ is a flat ring epimorphism. We claim that $Z = \{ \q \in \Spec R \mid \q S = S \}$ is a Thomason set. Indeed, for any $\q \in Z$ write $1=\sum_{i=1}^n a_is_i$ with $a_i\in \q$ and $s_i\in S$. Then the ideal $I=(a_1, \dots, a_n) \subseteq R$ is such that $IS=S$ and $I\subseteq \q$. Thus, $\q\in \Supp R/I \subseteq Z$ and $\Supp R/I$ is Zariski closed with quasi-compact complement. This proves the claim.

Let now $\X=\Ker (-\otimes_RS)$ be the localizing class corresponding to $S$ (see Theorem~\ref{thm:smashing-classif}). Then, using Notation~\ref{nota:flat-epi-wgldim1} and Lemma~\ref{lem:non-hom-support}, we have
\begin{multline*}
\X = \big\{ X \in \Der R \mid H^n(X) \otimes_R R_{s(\m)} = 0 \textrm{ for all } n \in \bbZ \textrm{ and } \m \in \Max R \big\} = \\
   = \{ X \in \Der R \mid \Supp X \subseteq Z \}.
\end{multline*}
Thus, $S$ describes the compactly generated localization corresponding to $Z$ by Proposition~\ref{prop:Thomason-loc-wgldim1}.

The last part concerning universal localizations follows from Proposition~\ref{prop:univ-loc-wd1}.
\end{proof}

% ------------------------------------------------------------------------------
\section{The Telescope Conjecture}
\label{sec:tc}

Finally, we will discuss the Telescope Conjecture for rings of weak global dimension $\le 1$. Although we do not obtain a full classification of smashing localizations as in the case of valuation domains, we are still able to obtain an easy criterion characterizing when the Telescope Conjecture holds for $\Der R$. In particular we will see that this is always the case when $R$ is a commutative von Neumann regular (also known as absolutely flat) ring, generalizing~\cite[Theorem 4.21]{Ste14}.

\begin{defn} \label{def:tc}
Let $\T$ be a triangulated category with coproducts. We say that the \emph{Telescope Conjecture} holds for $\T$ if every smashing localization of $\T$ is a compactly generated localization (see \S\ref{subsec:TTF-triples}).
\end{defn}

In fact, the Telescope Conjecture is a property of $\T$, it holds for some triangulated categories and fails for others. For $\Der R$ with $R$ commutative and $\wgldim R \le 1$, it asks for every homological epimorphism $f\dd R \to S$ to be flat. If $R$ is even semihereditary, it equivalently requires that every homological epimorphism $f\dd R \to S$ be a universal localization (see also~\cite[\S\S6 and 7]{KrSt}). Now we can state the main result of the final section.

\begin{thm} \label{thm:tc}
Let $R$ be a commutative ring of weak global dimension $\le 1$. Then the following are equivalent:
\begin{enumerate}
 \item The Telescope Conjecture holds for $\Der R$;
 \item Every homological epimorphism $f\dd R \to S$ is flat;
 \item There is no $\p \in \Spec R$ \st $\p R_\p$ is a non-zero idempotent ideal in~$R_\p$.
\end{enumerate}
\end{thm}

\begin{proof}
(1) $\Leftrightarrow$ (2) follows from Theorem~\ref{thm:flat-epi}. Assuming (2), let $\p \in \Spec R$. If $0 \ne \p R_\p$ is idempotent in $R_\p$, then $R \to R_\p/\p R_\p$ is a non-flat homological epimorphism by Lemma~\ref{lem:kernel-idempotent}, hence (2) $\Rightarrow$ (3). Finally, assume (3) and let $f\dd R \to S$ be a homological epimorphism. Then $f \otimes_S R_\p\dd R_\p \to S_\p$ must be flat for each $\p \in \Spec R$ by Theorem~\ref{thm:hom-epi-vd}. Hence $f$ is flat and (2) follows.
\end{proof}

In particular, we have the following necessary condition.

\begin{cor} \label{cor:acc-primes}
If $R$ is commutative, $\wgldim R \le 1$ and the Telescope Conjecture holds for $\Der R$, then $(\Spec R, \subseteq)$ has the ascending chain condition on prime ideals.
\end{cor}

\begin{proof}
If there is an infinite chain $\p_0 \subseteq \p_1 \subseteq \p_2 \subseteq \cdots$ of primes of $R$, then $\p = \bigcup_i \p_i$ is also a prime ideal which is necessarily idempotent in its localization by Lemmas~\ref{lem:locally-VD} and~\ref{lem:lots-of-idempotents}.
\end{proof}

We end by listing some classes of commutative semihereditary rings $R$ studied in the literature such that $\Der R$ satisfies the telescope conjecture.

\begin{enumerate}
\item Recall that a commutative domain is a \emph{Pr\"ufer domain} if every localization at a maximal (or prime) ideal is a valuation domain. A Pr\"ufer domain is \emph{strongly discrete}~\cite[\S III.7]{FS01} if no non-zero prime ideal is idempotent. Then $\Der R$ satisfies the Telescope Conjecture for a Pr\"ufer domain $R$ \iff $R$ is strongly discrete; see~\cite[Proposition III.7.4]{FS01}. 

\item If $R$ a commutative von Neumann regular ring, then $\Der R$ satisfies the telescope conjecture. This generalizes~\cite[Theorem 4.21]{Ste14}. In fact, $R$ is semihereditary and every localization at a maximal ideal is a field (see e.g. \cite[Corollary 4.2.7]{Glaz89}). Note that in this case every ideal of $R$ is idempotent.
 \end{enumerate}

% --------------------------------------
\appendix

\section{The homotopy categories of dg modules and algebras}
\label{sec:homotopy-dg}

Here we collect some background material on model categories and especially on model structures on the categories of dg modules and dg algebras. Although the presented results are certainly well-known to experts, we rely on precise manipulation with dg algebras and also dg bimodules at certain places and it seems convenient to have the necessary background included for the sake of completeness. The material presented here supplements mostly Sections~\ref{sec:hom-epi-dga} and~\ref{sec:hom-epi-semihered}. 

Although for additive categories like the one of dg modules over a fixed dg algebra the model category machinery can be circumvented to quite some extent by using the standard algebraic calculus of fractions as in~\cite{Kel94}, this does not a priory apply to the category of dg algebras which is \emph{not} an additive category. Formally inverting the class of quasi-isomorphisms of dg algebras is for example important to give a more conceptual definition of a homological epimorphism in Section~\ref{sec:hom-epi-dga}. The language of model categories~\cite{Hir,Hov99} is a very classical language allowing one to manipulate with localizations of categories in a systematic way, and it has the advantage that various constructions which do not make sense in the language of triangulated categories (like general homotopy limits and colimits) have a precise meaning there.

% --------------------------------------
\subsection{Recollections of model categories}
\label{subsec:model-cat}

Starting with a category $\A$ and a class $\we$ of morphisms in $\A$, on can always construct (up to a possible set theoretic difficulty) a universal functor $Q\colon \A \to \A[\we\inv]$ which makes the maps in $\we$ invertible; see~\cite[\S1.1]{GZ67}. Being able to work with $\A[\we\inv]$ efficiently is, however, a much harder task. The classical topologically motivated idea behind Quillen's model structures~\cite{Hir,Hov99} allows one to solve this problem if we manage to find more structure, namely two additional classes of morphisms called fibrations and cofibrations, satisfying certain axioms which we recall below. A key concept in the definition is the following one.

\begin{defn} \label{def:wfs}
A pair $(\clL,\R)$ of classes of morphisms in a category $\A$ is a \emph{\wfs}if
\begin{enumerate}
\item $\clL$ and $\R$ are closed under retracts.

\item For any commutative square given by the solid arrows
\[
\xymatrix{
A \ar[r] \ar[d]_f     & X\phantom{,} \ar[d]^g      \\
B \ar[r] \ar@{.>}[ur] & Y,
}
\]
with $f \in \clL$ and $g \in \R$, there exists a (not necessarily unique) diagonal dotted morphism \st both triangles commute.

\item For every morphism $h\colon X \to Y$ in $\A$, there is a factorization
\[
\xymatrix@R=1em{
X \ar[rr]^h \ar@/_/[dr]_f && Y  \\
& Z \ar@/_/[ur]_g
}
\]
with $f \in \clL$ and $g \in \R$. 
\end{enumerate}
\end{defn}

\begin{rem} \label{rem:wfs-abelian}
If $\A$ is an abelian category, there is an efficient way to construct \wfss from so-called cotorsion pairs. This is in fact a useful technique to verify axioms of a model category for the categories of complexes and dg modules discussed below. We refer the reader to~\cite{Hov02,G3,Hov07,Beck14,St14} for details and many more further references.
\end{rem}

\begin{defn}[{\cite[\S1.1]{Hov99} or~\cite[\S7.1]{Hir}}] \label{def:models-str}
A \emph{model structure} on a category $\A$ is a triple of classes of morphisms $(\cof,\we,\fib)$ \st
\begin{enumerate}
\item $(\cof,\we\cap\fib)$ and $(\cof\cap\we,\fib)$ are weak factorization systems, and
\item $\we$ is closed under retracts and satisfies the 2-out-of-3 property for compositions (i.e.\ if $f$, $g$ are composable morphisms in $\A$ and two of $f$, $g$, $gf$ are in $\we$, so is the third).
\end{enumerate}

A model category is a bicomplete category $\A$ together with a chosen model structure $(\cof,\we,\fib)$.
\end{defn}

Regarding the terminology, the morphisms in the classes $\cof,\we,\fib$ are called \emph{cofibrations}, \emph{weak equivalences} and \emph{fibrations}, respectively. The morphisms in the intersections $\cof\cap\we$ and $\fib\cap\we$ are called \emph{trivial cofibrations} and \emph{trivial fibrations}, respectively. An object $C \in \A$ is \emph{(trivially) cofibrant} if the morphism $0 \to C$ from the initial object $0 \in \A$ is a (trivial) cofibration. If $X \in \A$ is an arbitrary object, we can apply the axiom Definition~\ref{def:wfs}(3) to $0 \to X$ to obtain a trivial fibration $p\colon C(X) \to X$ with $C(X)$ cofibrant. Such $C(X)$ is called the \emph{cofibrant replacement} of $X$ and it plays the role of a suitable resolution of $X$. Dually we can define \emph{(trivially) fibrant objects} and \emph{fibrant replacements}.

In Section~\ref{sec:hom-epi-dga} we also need the homotopy relation on morphisms of dg algebras. There is in fact a general notion which works for any model category.

\begin{defn}[{\cite[\S1.2]{Hov99} or~\cite[\S7.3]{Hir}}] \label{def:htpy}
A \emph{cylinder object} in $\A$ is a factorization of the fold map $1_X+1_X\colon X \amalg X \to X$ for an object $X \in \A$ into a cofibration followed by a weak equivalence,
\[
X \amalg X \mapr{i_0+i_1} D \mapr{\tau} X
\]
Two maps $f,g\colon X \to Y$ are \emph{left homotopic} is there exists a cylinder object and a map (the left homotopy) $h\colon D \to Y$ \st $hi_0 = f$ and $hi_1 = g$. Dually one defines \emph{path objects} and the \emph{right homotopy} relation.

Two maps $f,g\colon X \to Y$ are \emph{homotopic}, $f \sim g$ in symbols, if they are both left and right homotopic.
\end{defn}

In general, the left and right homotopy relations do not agree and may not be transitive. If $f\colon X \to Y$ is a map from a cofibrant to a fibrant object, however, then the left and right homotopy on $\A(X,Y)$ \emph{do} agree and they are equivalence relations; see~\cite[Proposition 1.2.5]{Hov99}.

The localized category $\A[\we\inv]$ of a model category $\A$ is traditionally called the \emph{homotopy category} and denoted by $\Ho\A$. For the sake of completeness, we restate the following classical result. 

\begin{prop} \label{prop:model-basic}
Let $\A$ together with $(\cof,\we,\fib)$ be a model category, let $Q\colon \A \to \Ho\A$ be the canonical localization functor, and denote by $\A_{cf}$ the full subcategory of $\A$ formed by the fibrant and cofibrant objects.

\begin{enumerate}
\item If $C \in \A$ is cofibrant and $F \in \A$ is fibrant, then $Q$ induces a bijection
\[ \big(\A(C,F)/\!\sim\!\big) \to \Ho\A\big(Q(C),Q(F)\big). \]
In particular $\Ho\A\big(Q(X),Q(Y)\big) \cong \A\big(C(X),F(Y)\big)/\!\!\sim$ for any pair of objects $X,Y\in\A$.

\item The homotopy relation on morphisms of $\A_{cf}$ is compatible with composition of morphisms, and $\A_{cf} \mapr{\subseteq} \A \mapr{Q} \Ho\A$ induces an equivalence of categories $(\A_{cf}/\!\sim\!) \to \Ho\A$.

\item If $f\colon X \to Y$ is a morphism in $\A$, then $Q(f)$ is an isomorphism in $\Ho\A$ \iff $f$ is a weak equivalence. 
\end{enumerate}
\end{prop}

\begin{proof}
These facts are dating back to \cite{QHtp}. We refer to \cite[Theorem 1.2.10]{Hov99} of \cite[\S8.4]{Hir} for a proof.
\end{proof}

\begin{rem} \label{rem:morphisms-in-model-cat}
It follows from Proposition~\ref{prop:model-basic}(1) that any morphism $g\colon X \to Y$ in $\Ho\A$ can be represented by a zigzag of morphism in $\A$,
\[ \xymatrix@1{X & C(X) \ar[l] \ar[r] & F(Y) & Y \ar[l]}, \]
where the first and the last one are cofibrant and fibrant replacements, respectively. In particular, these are invertible in $\Ho\A$.

If it so happens that all objects of $\A$ are fibrant (as in Propositions~\ref{prop:model-dg-mod} and~\ref{prop:model-dgas} below), then each morphism of $\Ho\A$ is even represented by a cospan
\[
\xymatrix{
& C(X) \ar[dl]_\sigma \ar[dr]^f \\
X && Y,
}
\]
in $\A$. That is, $g = Q(f) Q(\sigma)\inv$. We will usually simply write $g=f\sigma\inv$ by slightly abusing notation.

It is also important to know when two such fractions $f\sigma\inv, f'(\sigma')\inv\colon X \to Y$, where $\sigma\colon C \to X$ and $\sigma'\colon C' \to X$ are cofibrant replacements of $X$ in $\A$, are equal in the homotopy category $\Ho\A$. Given two such fractions, $\sigma$ factors through $\sigma'$ by Definition~\ref{def:wfs}(2), using that $(\cof,\we\cap\fib)$ is a weak factorization system. Hence we get a diagram in $\A$ such that the left hand side triangle commutes and $\sigma,\sigma'$ and $\sigma''$ are weak equivalences:
\[
  \xymatrix@R=1em{
	  && C \ar[dll]_-\sigma \ar@{.>}[dd]^-{\sigma''} \ar[drr]^-f \\
	  X &&&& Y \\
	  && C' \ar[ull]^-{\sigma'} \ar[urr]_-{f'} \\
  }
\]
In particular, $f\sigma\inv = f'(\sigma')\inv$ in $\Ho\A$ \iff $f = f'\sigma''\colon C \to Y$ in $\Ho\A$. As both the maps are maps from a cofibrant to a fibrant object, we obtain the following characterization from Proposition~\ref{prop:model-basic}(1):

The fractions $f\sigma\inv$ and $f'(\sigma')\inv$ are equal in $\Ho\A$ \iff $\sigma$ factors as $\sigma = \sigma'\sigma''$ in such a way that $\sigma',\sigma''$ are weak equivalences and $f \sim f'\sigma''$ are homotopic in the model category $\A$.
\end{rem}

% --------------------------------------
\subsection{Categories of complexes}
\label{subsec:catg-cpx}

Our aim is to apply the above abstract theory to dg modules or dg algebras. For this we need a good grip of categories of complexes of modules, especially over commutative rings. Our approach is inspired by (the appendices of) \cite{Av13}.

To start with, let $R$ be a ring and denote by $\Cpx R$ the category of cochain complexes of right $R$-modules. We use the cohomological indexing for components of a complex $X$,
\[ X\colon \qquad \cdots \la X^0 \mapr{\dif^0_X} X^1 \mapr{\dif^1_X} X^2 \la \mapr{\dif^2_X} X^3 \la \cdots. \]
It is well known (see~\cite[Theorem 2.3.11]{Hov99}) that $\Cpx R$ carries a model structure such that:
\begin{enumerate}
 \item Weak equivalences are the quasi-isomorphisms.
 \item Fibrations are precisely the epimorphisms in $\Cpx R$ (i.e.\ the maps of complexes which are componentwise surjective).
\end{enumerate}

In particular every object is fibrant. Moreover, cofibrations are precisely mono\-mor\-phisms with cofibrant cokernel, a cofibrant object has all components projective, and trivially cofibrant objects are precisely the projective objects in $\Cpx R$. Various names are used for cofibrant objects in this context: $K$-projective complexes~\cite{Spa88}, complexes with property~(P)~\cite{Kel94}, homotopically projective complexes~\cite{Kel98} and probably several others. We will use the term \emph{homotopically projective} here.

The \emph{unbounded derived category} of $R$, denoted by $\Der R$, is by definition the homotopy category $\Ho{\Cpx R}$, in the sense of~\cite{Hir,Hov99}, of the model category~$\Cpx R$.

Suppose now that $R,S,T$ are rings, $X_R$ is a complex of right $R$-modules, ${_RY_S}$ is a complex of $R$-$S$-bimodules and $Z_S$ is a complex of right $S$-modules. Then we can define the tensor product $X \otimes_R Y \in \Cpx S$ in the usual way. That is,
\[ (X \otimes_R Y)^i = \bigoplus_{p+q = i} X^p \otimes_R Y^q \]
and the differential $\dif^i_{X \otimes_R Y}\dd (X \otimes_R Y)^{i} \la (X \otimes_R Y)^{i+1}$ is defined using the graded Leibniz rule, so that for $x \in X^p$ and $y \in Y^q$ with $p+q = i$ we have
\begin{equation} \label{eqn:tensor-dif}
\dif_{X \otimes_R Y}(x \otimes y) = \dif_X(x) \otimes y + (-1)^p x \otimes \dif_Y(y). 
\end{equation}
It is straightforward to check that this is well-defined and that $\dif^2 = 0$.

Similarly we can define the internal Hom-functor. We define $\HOM_S(Y,Z) \in \Cpx R$ so that
\[ \HOM_S(Y,Z)^i = \prod_{p \in \bbZ} \Hom_S(Y^p,Z^{p+i}) \]
for each $i \in \bbZ$, and the differential is defined as the graded commutator. That is, if $f = (f^p)_{p \in \bbZ} \in \prod_{p \in \bbZ} \Hom_S(Y^p,Z^{p+i})$ is a collection of morphisms of $S$-modules, we put
\[ \dif_{\HOM_S(Y,Z)}(f) = \dif_Z \circ f - (-1)^i f \circ \dif_Y. \]

The usual adjunction between the $\Hom$ and tensor functors for modules extends in a completely straightforward way to an analogous isomorphism of complexes of abelian groups
\[ \HOM_S(X \otimes_R Y, Z) \cong \HOM_R(X, \HOM_S(Y,Z)), \]
which is natural in all three variables. In fact, all this can be done much more abstractly, see for instance~\cite[Appendix A]{HPS97}.

If we start with a commutative ring $k$, then the above specialize to functors
\[
\otimes_k\dd \Cpx k \times \Cpx k \la \Cpx k
\qquad \textrm{and} \qquad
\HOM_k\dd {\Cpx k}\op \times \Cpx k \to \Cpx k
\]
and provides us with a closed symmetric monoidal structure on $\Cpx k$ in the sense of~\cite[\S VII.7]{McL2}. A little care is due when defining the commutativity isomorphisms $\gamma_{X,Y}\dd X \otimes Y \to Y \otimes X$ as we need to introduce the so-called Koszul signs. If $x \in X^p$ and $y \in Y^q$, then $\gamma_{X,Y}(x \otimes y) = (-1)^{pq} y \otimes x$. The tensor unit is $k$ itself viewed as a complex concentrated in degree zero.

A \emph{dg algebra} $A$ over $k$ is defined as a monoid in $\Cpx k$ in the sense of~\cite[\S VII.3]{McL2}. Strictly speaking, we should write $(A,\mu,\eta)$ instead of just $A$, where $\mu\dd A \otimes_k A \to A$ and $\eta\dd k \to A$ are morphisms in $\Cpx k$, but we as usual view these as implicitly given. In more pedestrian terms, $A$ is a $\bbZ$-graded $k$-algebra with a differential of degree $1$ which satisfies the graded Leibniz rule \wrt multiplication. A left or right \emph{dg module} $M$ over $A$ is a complex $M \in \Cpx k$ together with a (left or right) action of $A$ in the sense of~\cite[\S VII.4]{McL2}. A \emph{dg $A$-$B$-bimodule} is a complex $M$ with left dg $A$-module and right dg $B$-module structures which are compatible via the obvious associativity $(a\cdot m)\cdot b = a\cdot (m\cdot b)$ for each $a \in A$, $m \in M$ and $b \in B$.

We denote the category of dg algebras over $k$ by $\Dga k$ and, following the notation from~\cite{Kel94}, the category of right dg modules over a given dg algebra $A$ will be denoted by $\Cpx A$. Note that if we view an ordinary $k$-algebra $R$ as a dg algebra concentrated in degree 0, then the category of dg modules over $R$ is none other than the category of complexes of $R$-modules---that is $\Cpx R$ is the same in both senses. If $A$, $B$ are dg algebras, a dg $A$-$B$-bimodule can be viewed as module over $A \otimes_k B\op$, where the multiplication in $A \otimes_k B\op$ involves the corresponding Koszul signs (see~\cite[Appendix B]{Av13}).

% --------------------------------------
\subsection{A projective model structure on dg modules}
\label{subsec:model-dg-mod}

Now we shall discuss a model structure for the derived category, in the sense of~\cite{Kel94}, of a fixed dg algebra $A$. As we shall explain, an important added value of our approach here is that the manipulation with various derived functors becomes quite transparent.

\begin{prop} \label{prop:model-dg-mod}
Let $A \in \Dga k$. Then $\Cpx A$ admits model structures \st
\begin{enumerate}
 \item Weak equivalences are the quasi-isomorphisms.
 \item Fibrations are the surjective maps.
\end{enumerate}
\end{prop}

\begin{proof}
This is a special case of~\cite[Theorem 4.1]{SchSh00} with $\Cpx k$ as the underlying monoidal model category. Another incarnation can be found in \cite[Proposition 1.3.5(1)]{Beck14}.
\end{proof}

The \emph{derived category} $\Der A$ of $A$ is then by definition the homotopy category $\Ho\Cpx A$ \wrt the above model structure. Every object in this model category is fibrant and we will again call the cofibrant objects in $\Cpx A$ \emph{homotopically projective}. According to Proposition~\ref{prop:model-basic}(2), $\Der A$ is equivalent to the category of homotopically projective dg modules modulo the homotopy relation. Moreover, two maps of dg modules are homotopic in the sense of Definition~\ref{def:htpy} \iff they are homotopic in the sense of~\cite[\S2.1]{Kel94} (see~\cite[Proposition 1.1.14]{Beck14}). We thus obtain the following description for homotopically projective dg modules:

\begin{prop} \label{prop:htp-proj-dgmod}
Let $A$ be a dg algebra over $k$ and $X \in \Cpx A$ be a right dg module over $A$. Then $X$ is homotopically projective \iff $X$ is a summand in a dg module $P$ \st $P$ is the union $P = \bigcup_{i \ge 0} P_i$ of a chain
\[ 0 = P_0 \subseteq P_1 \subseteq P_2 \subseteq P_3 \subseteq \cdots \]
of dg submodules \st $P_{i+1}/P_i$ is for each $i \in \bbN$ a direct sum of copies of suspensions of $A$.
\end{prop}

\begin{proof}
The above discussion and \cite[\S3.1]{Kel94} imply that homotopically projective dg modules are precisely those `with property (P)' in the terminology of \cite[\S3.1]{Kel94}. The above description is precisely the definition of the `property (P)'. We also refer to~\cite[\S2.3]{StPo} for a more detailed argumentation.
\end{proof}

%In particular, we also have:
%
%\begin{cor}
%Let $X \in \Cpx A$ be homotopically projective. Then $X$ is projective when viewed only as a $\bbZ$-graded module over $A$.
%\end{cor}

Typically, one does not work only with $\Der A$ of a single dg algebra $A$, but one considers functors between derived categories of various dg algebras. These functors almost always arise as total derived functors of $\Hom$ and tensor functors represented by dg bimodules.

To this end, suppose that $A,B$ are dg algebras, $X_A$ is a right dg $A$-module, ${_AY_B}$ is a dg $A$-$B$-bimodule and that $Z_B$ is a right dg $B$-module. Then we can define $X \otimes_A Y \in \Cpx B$ and $\HOM_B(Y,Z) \in \Cpx A$, extending the definitions for complexes over ordinary algebras. Forgetting the differential and the right action of $B$ for the moment, the underlying $\bbZ$-graded $k$-module of $X \otimes_A Y$ is obtained as the tensor product $X \otimes_A Y$ of graded modules over $A$ as a graded algebra. The differential is again defined by formula~(\ref{eqn:tensor-dif}). In fact, $X \otimes_A Y$ is a factor of $X \otimes_k Y$ as a right dg $B$-module, see~\cite[Exercise VII.4.6, p. 175]{McL2}. Similarly $\HOM_B(Y,Z)$ is a $k$-subcomplex of $\HOM_k(Y,Z)$ \st $\HOM_B(Y,Z)^i$ consists of degree $i$ graded $B$-module homomorphisms $Y \to Z$; see~\cite[p. 499]{SchSh00}. Again we have an adjunction (natural isomorphism in $\Cpx k$)
\[ \HOM_B(X \otimes_A Y, Z) \cong \HOM_A(X, \HOM_B(Y,Z)). \]

For a general definition of total derived functors we refer to~\cite[\S8.4]{Hir}. For our purpose, it is enough to recall the following recipe based on \cite[Proposition 8.4.8]{Hir}: Given $A \in \Dga k$, a right dg module $X_A$ and a left dg module ${_AY}$, then there are canonical isomorphisms in $\Der k$,
\[ P \otimes_A Y \cong X \Lotimes_A Y \cong X \otimes_A Q, \]
where $\Lotimes$ is the total left derived functor of $\otimes\colon \Cpx{A}\times \Cpx{A\op} \to \Cpx{k}$, and $P_A \to X_A$ and ${_AQ} \to {_AY}$ are homotopically projective replacement of dg modules. If ${_AY_B}$ is a dg $A$-$B$-bimodule, the above isomorphisms also exist in $\Der B$. The second part of the following proposition is then particularly useful when we try to evaluate the composition of two left derived tensor functors.

\begin{lem} \label{lem:htp-proj-tensor} Let $A,B$ be dg algebras over $k$. Then the following hold:
\begin{enumerate}
\item Suppose that a dg module $P_A$ can be written as a union $P = \bigcup_{i \ge 0} P_i$ of a chain of dg $A$-submodules \st $P_0 = 0$ and each $P_{i+1}/P_i$ is homotopically projective in $\Cpx A$. Then also $P$ is homotopically projective.

\item Suppose that ${_AX_B}$ is a dg $A$-$B$-bimodule and ${_BY}$ is a dg $B$-module. If both ${_AX}$ and ${_BY}$ are homotopically projective as left dg modules, then so is ${_A(X \otimes_B Y)}$. 
\end{enumerate}
\end{lem}

\begin{proof}
(1) This follows from the description of homotopically projective dg modules in~\cite[Remark 2.15]{StPo}.

(2) We know that $Y$ is a summand of $P = \bigcup_{i \ge 0} P_i$ where $P_0 = 0$ and each $P_{i+1}/P_i$ is a coproduct of copies of suspensions of ${_BB}$. Since each $P_i \subseteq P_{i+1}$ splits as a map of graded $B$-modules and the underlying graded $A$-modules of $X \otimes_B P_i$ are simply the corresponding tensor products of graded modules over $B$, the chain
\[ 0 = X \otimes_B P_0 \la X \otimes_B P_1 \la X \otimes_B P_2 \la X \otimes_B P_3 \la \cdots \]
consists of monomorphism. Indeed, all these morphisms are split monomorphisms of graded $A$-modules. Moreover, each factor $X \otimes_B P_{i+1}/P_i$ is isomorphic to a coproduct of suspensions of $X$, hence is homotopically projective. Thus, $X \otimes_B Y$ is a homotopically projective dg $A$-module by part (1).
\end{proof}

% --------------------------------------
\subsection{A projective model structure on dg algebras}

There is also a model structure on the category $\Dga k$ of dg algebras over $k$ whose knowledge, however, does not seem to be so widely spread among algebraists. Now we need the power of the formalism of model categories since $\Dga k$ is far from being an additive category.

\begin{prop} \label{prop:model-dgas} \cite{Jard95,SchSh00}
Let $k$ be a commutative ring. Then $\Dga k$ admits a model structure \st
\begin{enumerate}
 \item Weak equivalences are the quasi-isomorphisms of dg algebras.
 \item Fibrations are the surjective maps of dg algebras.
\end{enumerate}
In particular every dg algebra is fibrant. Moreover, if $f\dd A \to B$ is a cofibration of dg algebras and the underlying complex of $A$ is homotopically projective in $\Cpx k$, then so is the underlying complex of $B$. If $A$ is a cofibrant dg algebra, this in particular means that the underlying complex of $A$ is homotopically projective in $\Cpx k$.
\end{prop}

\begin{proof}
The model structure comes from~\cite{Jard95}, and the statement is a specialization of~\cite[Theorem 4.1(3)]{SchSh00} (see also~\cite[pp. 503--504]{SchSh00}).
\end{proof}

The second half of the proposition is useful when dealing with homotopical epimorphisms of dg algebras since some results of \cite{NS09} can be applied only to dg algebras whose underlying complexes of $k$-modules are homotopically projective. Here we see that this is not an essential restriction since every dg algebra is quasi-isomorphic to its cofibrant replacement which has this property (the authors are grateful to Pedro Nicol\'as for explaining them this point).

To conclude our discussion, suppose that $f\dd B \to A$ is a homomorphism of dg algebras, and that $X_A$ and ${_A Y}$ are dg $A$-modules. We can also view $X$ and $Y$ as dg $B$-modules via $f$, and we can apply the left derived tensor product functors both over $A$ and over $B$. A favorable fact, useful in~\S\ref{sec:hom-epi-dga}, is that we obtain the same result as long as $A$ and $B$ are isomorphic in $\Ho{\Dga k}$.

\begin{lem} \label{lem:dga-change}
Let $f\dd B \to A$ be a homomorphism of dg algebras which is a quasi-isomorphism, and let $X_A$ and ${_A Y}$ be dg $A$-modules. Then there is a natural isomorphism $X \Lotimes_A Y \cong X \Lotimes_B Y$ in $\Der k$.
\end{lem}

\begin{proof}
The key claim is that if $p\dd P_B \to X_B$ is a homotopically projective replacement of $X$ as a dg $B$-module, then the composition of $p \otimes_B A\dd P \otimes_B A \to X\otimes_B A $ with the multiplication $\mu\dd X\otimes_B A\to X$ is a homotopically projective replacement of $X$ as a dg $A$-module. Indeed, since $p$ is surjective, so is $\mu\circ (p \otimes_B A)$, hence $\mu\circ (p \otimes_B A)$ is a fibration in $\Cpx A$. Further, $P \otimes_B A$ is homotopically projective over~$A$ by Lemma~\ref{lem:htp-proj-tensor}(2).

It remains to prove that $\mu\circ (p \otimes_B A)$ is a quasi-isomorphism. To this end, note that $p$ can be written as the composition
\[
\xymatrix{
P \ar[rr]^p \ar[d]_{P \otimes_B f} && X  \\
P \otimes_B A \ar[rr]^{p \otimes_B A} && X \otimes_B A \ar[u]_\mu. \\
}
\]
Since $p$ is a quasi-isomorphism to start with, our task is equivalent to proving that $P \otimes_B f$ is a quasi-isomorphism. However, the latter follows from the fact that $f$ is a quasi-isomorphism of left dg $B$-modules and $P_B$ is homotopically projective as a right dg $B$-module. This proves the claim.

Now we have $X  \Lotimes_B Y = P \otimes_B Y$ and 
$X \Lotimes_A Y = (P \otimes_B A) \otimes_A Y$ (see~\cite[Proposition 8.4.8]{Hir})
and the right hand sides are naturally isomorphic. 
\end{proof}

% ----------------------------------------------------------------------------
% The bibliography
\bibliographystyle{alpha}
\bibliography{references}

\begin{thebibliography}{AHKL11}

\bibitem[AHKL11]{AKL11}
Lidia Angeleri~H{\"u}gel, Steffen Koenig, and Qunhua Liu.
\newblock Recollements and tilting objects.
\newblock {\em J. Pure Appl. Algebra}, 215(4):420--438, 2011.

\bibitem[Avr13]{Av13}
Luchezar~L. Avramov.
\newblock ({C}ontravariant) {K}oszul duality for {DG} algebras.
\newblock In {\em Algebras, quivers and representations}, volume~8 of {\em Abel
  Symp.}, pages 13--58. Springer, Heidelberg, 2013.

\bibitem[Bal05]{Bal05}
Paul Balmer.
\newblock The spectrum of prime ideals in tensor triangulated categories.
\newblock {\em J. Reine Angew. Math.}, 588:149--168, 2005.

\bibitem[BBD82]{BBD}
A.~A. Be{\u\i}linson, J.~Bernstein, and P.~Deligne.
\newblock Faisceaux pervers.
\newblock In {\em Analysis and topology on singular spaces, {I} ({L}uminy,
  1981)}, volume 100 of {\em Ast\'erisque}, pages 5--171. Soc. Math. France,
  Paris, 1982.

\bibitem[Bec14]{Beck14}
Hanno Becker.
\newblock Models for singularity categories.
\newblock {\em Adv. Math.}, 254:187--232, 2014.

\bibitem[BF11]{BaFa11}
Paul Balmer and Giordano Favi.
\newblock Generalized tensor idempotents and the telescope conjecture.
\newblock {\em Proc. Lond. Math. Soc. (3)}, 102(6):1161--1185, 2011.

\bibitem[BIK08]{BIK08}
Dave Benson, Srikanth~B. Iyengar, and Henning Krause.
\newblock Local cohomology and support for triangulated categories.
\newblock {\em Ann. Sci. \'Ec. Norm. Sup\'er. (4)}, 41(4):573--619, 2008.

\bibitem[BIK11]{BIK11}
Dave Benson, Srikanth~B. Iyengar, and Henning Krause.
\newblock Stratifying triangulated categories.
\newblock {\em J. Topol.}, 4(3):641--666, 2011.

\bibitem[BK72]{BK72}
A.~K. Bousfield and D.~M. Kan.
\newblock The core of a ring.
\newblock {\em J. Pure Appl. Algebra}, 2:73--81, 1972.

\bibitem[BK73]{BK72-corr}
A.~K. Bousfield and D.~M. Kan.
\newblock Correction to ``{T}he core of a ring'' ({J}. {P}ure {A}ppl. {A}lgebra
  {\bf 2} (1972), 73--81).
\newblock {\em J. Pure Appl. Algebra}, 3:409, 1973.

\bibitem[BO95]{BO95}
Alexey~I. Bondal and Dmitri~O. Orlov.
\newblock Semiorthogonal decomposition for algebraic varieties.
\newblock Preprint, arXiv:alg-geom/9506012v1, 1995.

\bibitem[Bou79]{Bou79}
A.~K. Bousfield.
\newblock The localization of spectra with respect to homology.
\newblock {\em Topology}, 18(4):257--281, 1979.

\bibitem[CE56]{CE56}
Henri Cartan and Samuel Eilenberg.
\newblock {\em Homological algebra}.
\newblock Princeton University Press, Princeton, N. J., 1956.

\bibitem[DHS88]{DHS88}
Ethan~S. Devinatz, Michael~J. Hopkins, and Jeffrey~H. Smith.
\newblock Nilpotence and stable homotopy theory. {I}.
\newblock {\em Ann. of Math. (2)}, 128(2):207--241, 1988.

\bibitem[DP08]{DwPa08}
W.~G. Dwyer and J.~H. Palmieri.
\newblock The {B}ousfield lattice for truncated polynomial algebras.
\newblock {\em Homology Homotopy Appl.}, 10(1):413--436, 2008.

\bibitem[EJ00]{EJ00}
Edgar~E. Enochs and Overtoun M.~G. Jenda.
\newblock {\em Relative homological algebra}, volume~30 of {\em de Gruyter
  Expositions in Mathematics}.
\newblock Walter de Gruyter \& Co., Berlin, 2000.

\bibitem[FS01]{FS01}
L{\'a}szl{\'o} Fuchs and Luigi Salce.
\newblock {\em Modules over non-{N}oetherian domains}, volume~84 of {\em
  Mathematical Surveys and Monographs}.
\newblock American Mathematical Society, Providence, RI, 2001.

\bibitem[GdlP87]{GdlP87}
P.~Gabriel and J.~A. de~la Pe{\~n}a.
\newblock Quotients of representation-finite algebras.
\newblock {\em Comm. Algebra}, 15(1-2):279--307, 1987.

\bibitem[Gil04]{G3}
James Gillespie.
\newblock The flat model structure on {${\rm Ch}(R)$}.
\newblock {\em Trans. Amer. Math. Soc.}, 356(8):3369--3390 (electronic), 2004.

\bibitem[GL91]{GL91}
Werner Geigle and Helmut Lenzing.
\newblock Perpendicular categories with applications to representations and
  sheaves.
\newblock {\em J. Algebra}, 144(2):273--343, 1991.

\bibitem[Gla89]{Glaz89}
Sarah Glaz.
\newblock {\em Commutative coherent rings}, volume 1371 of {\em Lecture Notes
  in Mathematics}.
\newblock Springer-Verlag, Berlin, 1989.

\bibitem[Gla05]{Glaz05}
Sarah Glaz.
\newblock Pr\"ufer conditions in rings with zero-divisors.
\newblock In {\em Arithmetical properties of commutative rings and monoids},
  volume 241 of {\em Lect. Notes Pure Appl. Math.}, pages 272--281. Chapman \&
  Hall/CRC, Boca Raton, FL, 2005.

\bibitem[GZ67]{GZ67}
P.~Gabriel and M.~Zisman.
\newblock {\em Calculus of fractions and homotopy theory}.
\newblock Ergebnisse der Mathematik und ihrer Grenzgebiete, Band 35.
  Springer-Verlag New York, Inc., New York, 1967.

\bibitem[Hir03]{Hir}
Philip~S. Hirschhorn.
\newblock {\em Model categories and their localizations}, volume~99 of {\em
  Mathematical Surveys and Monographs}.
\newblock American Mathematical Society, Providence, RI, 2003.

\bibitem[Hov99]{Hov99}
Mark Hovey.
\newblock {\em Model categories}, volume~63 of {\em Mathematical Surveys and
  Monographs}.
\newblock American Mathematical Society, Providence, RI, 1999.

\bibitem[Hov02]{Hov02}
Mark Hovey.
\newblock Cotorsion pairs, model category structures, and representation
  theory.
\newblock {\em Math. Z.}, 241(3):553--592, 2002.

\bibitem[Hov07]{Hov07}
Mark Hovey.
\newblock Cotorsion pairs and model categories.
\newblock In {\em Interactions between homotopy theory and algebra}, volume 436
  of {\em Contemp. Math.}, pages 277--296. Amer. Math. Soc., Providence, RI,
  2007.

\bibitem[HPS97]{HPS97}
Mark Hovey, John~H. Palmieri, and Neil~P. Strickland.
\newblock Axiomatic stable homotopy theory.
\newblock {\em Mem. Amer. Math. Soc.}, 128(610):x+114, 1997.

\bibitem[Isb69]{Isb69}
John~R. Isbell.
\newblock Epimorphisms and dominions. {IV}.
\newblock {\em J. London Math. Soc. (2)}, 1:265--273, 1969.

\bibitem[Jar97]{Jard95}
J.~F. Jardine.
\newblock A closed model structure for differential graded algebras.
\newblock In {\em Cyclic cohomology and noncommutative geometry ({W}aterloo,
  {ON}, 1995)}, volume~17 of {\em Fields Inst. Commun.}, pages 55--58. Amer.
  Math. Soc., Providence, RI, 1997.

\bibitem[Kel94a]{Kel94}
Bernhard Keller.
\newblock Deriving {DG} categories.
\newblock {\em Ann. Sci. \'Ecole Norm. Sup. (4)}, 27(1):63--102, 1994.

\bibitem[Kel94b]{Kel94-smash}
Bernhard Keller.
\newblock A remark on the generalized smashing conjecture.
\newblock {\em Manuscripta Math.}, 84(2):193--198, 1994.

\bibitem[Kel98]{Kel98}
Bernhard Keller.
\newblock On the construction of triangle equivalences.
\newblock In {\em Derived equivalences for group rings}, volume 1685 of {\em
  Lecture Notes in Math.}, pages 155--176. Springer, Berlin, 1998.

\bibitem[Kel07]{Kel07}
Bernhard Keller.
\newblock Derived categories and tilting.
\newblock In {\em Handbook of tilting theory}, volume 332 of {\em London Math.
  Soc. Lecture Note Ser.}, pages 49--104. Cambridge Univ. Press, Cambridge,
  2007.

\bibitem[KP15]{KoPi13}
Joachim Kock and Wolfgang Pitsch.
\newblock Hochster duality in derived categories and point-free reconstruction
  of schemes.
\newblock Preprint, arXiv:1305.1503v5, 2015.

\bibitem[Kra00]{Kr00}
Henning Krause.
\newblock Smashing subcategories and the telescope conjecture---an algebraic
  approach.
\newblock {\em Invent. Math.}, 139(1):99--133, 2000.

\bibitem[Kra05]{Kr05-telescope}
Henning Krause.
\newblock Cohomological quotients and smashing localizations.
\newblock {\em Amer. J. Math.}, 127(6):1191--1246, 2005.

\bibitem[Kra07]{Kr07}
Henning Krause.
\newblock Derived categories, resolutions, and {B}rown representability.
\newblock In {\em Interactions between homotopy theory and algebra}, volume 436
  of {\em Contemp. Math.}, pages 101--139. Amer. Math. Soc., Providence, RI,
  2007.

\bibitem[Kra10]{Kr10}
Henning Krause.
\newblock Localization theory for triangulated categories.
\newblock In {\em Triangulated categories}, volume 375 of {\em London Math.
  Soc. Lecture Note Ser.}, pages 161--235. Cambridge Univ. Press, Cambridge,
  2010.

\bibitem[K{\v{S}}10]{KrSt}
Henning Krause and Jan {\v{S}}{\v{t}}ov{\'{\i}}{\v{c}}ek.
\newblock The telescope conjecture for hereditary rings via {E}xt-orthogonal
  pairs.
\newblock {\em Adv. Math.}, 225(5):2341--2364, 2010.

\bibitem[Laz69]{Laz69}
Daniel Lazard.
\newblock Autour de la platitude.
\newblock {\em Bull. Soc. Math. France}, 97:81--128, 1969.

\bibitem[Maz68]{Maz68}
Pierre Mazet.
\newblock Caract\'erisation des \'epimorphismes par relations et
  g\'en\'erateurs.
\newblock In {\em Les \'epimorphismes d'anneaux}, Sem. P. Samuel (2), pages
  A2.01--A2.08. Paris, 1968.

\bibitem[Mit73]{Mitch73}
Barry Mitchell.
\newblock The cohomological dimension of a directed set.
\newblock {\em Canad. J. Math.}, 25:233--238, 1973.

\bibitem[ML98]{McL2}
Saunders Mac~Lane.
\newblock {\em Categories for the working mathematician}, volume~5 of {\em
  Graduate Texts in Mathematics}.
\newblock Springer-Verlag, New York, second edition, 1998.

\bibitem[Nee92a]{Nee92}
Amnon Neeman.
\newblock The chromatic tower for {$D(R)$}.
\newblock {\em Topology}, 31(3):519--532, 1992.
\newblock With an appendix by Marcel B{\"o}kstedt.

\bibitem[Nee92b]{Nee92-loc}
Amnon Neeman.
\newblock The connection between the {$K$}-theory localization theorem of
  {T}homason, {T}robaugh and {Y}ao and the smashing subcategories of
  {B}ousfield and {R}avenel.
\newblock {\em Ann. Sci. \'Ecole Norm. Sup. (4)}, 25(5):547--566, 1992.

\bibitem[Nee00]{Nee00-oddball}
Amnon Neeman.
\newblock Oddball {B}ousfield classes.
\newblock {\em Topology}, 39(5):931--935, 2000.

\bibitem[Nee01]{Nee01}
Amnon Neeman.
\newblock {\em Triangulated categories}, volume 148 of {\em Annals of
  Mathematics Studies}.
\newblock Princeton University Press, Princeton, NJ, 2001.

\bibitem[NS09]{NS09}
Pedro Nicol{\'a}s and Manuel Saor{\'{\i}}n.
\newblock Parametrizing recollement data for triangulated categories.
\newblock {\em J. Algebra}, 322(4):1220--1250, 2009.

\bibitem[NS13]{NS13}
Pedro Nicol{\'a}s and Manuel Saor{\'{\i}}n.
\newblock Pseudoepimorphisms and semiorthogonal decompositions.
\newblock Private communication, 2013.

\bibitem[Pau09]{Pauk09}
David Pauksztello.
\newblock Homological epimorphisms of differential graded algebras.
\newblock {\em Comm. Algebra}, 37(7):2337--2350, 2009.

\bibitem[Qui67]{QHtp}
Daniel~G. Quillen.
\newblock {\em Homotopical algebra}.
\newblock Lecture Notes in Mathematics, No. 43. Springer-Verlag, Berlin, 1967.

\bibitem[Rav84]{Rav84}
Douglas~C. Ravenel.
\newblock Localization with respect to certain periodic homology theories.
\newblock {\em Amer. J. Math.}, 106(2):351--414, 1984.

\bibitem[Sch85]{Scho85}
A.~H. Schofield.
\newblock {\em Representation of rings over skew fields}, volume~92 of {\em
  London Mathematical Society Lecture Note Series}.
\newblock Cambridge University Press, Cambridge, 1985.

\bibitem[Sch05]{Sch12}
Karl Schwede.
\newblock Gluing schemes and a scheme without closed points.
\newblock In {\em Recent progress in arithmetic and algebraic geometry}, volume
  386 of {\em Contemp. Math.}, pages 157--172. Amer. Math. Soc., Providence,
  RI, 2005.

\bibitem[Sil67]{Sil67}
L.~Silver.
\newblock Noncommutative localizations and applications.
\newblock {\em J. Algebra}, 7:44--76, 1967.

\bibitem[{\v{S}}P16]{StPo}
Jan {\v{S}}{\v{t}}ov{\'{\i}}{\v{c}}ek and David Posp{\'{\i}}{\v{s}}il.
\newblock On compactly generated torsion pairs and the classification of
  co-{$t$}-structures for commutative noetherian rings.
\newblock {\em Trans. Amer. Math. Soc.}, 368(9):6325--6361, 2016.

\bibitem[Spa88]{Spa88}
N.~Spaltenstein.
\newblock Resolutions of unbounded complexes.
\newblock {\em Compositio Math.}, 65(2):121--154, 1988.

\bibitem[SS00]{SchSh00}
Stefan Schwede and Brooke~E. Shipley.
\newblock Algebras and modules in monoidal model categories.
\newblock {\em Proc. London Math. Soc. (3)}, 80(2):491--511, 2000.

\bibitem[Ste75]{Sten75}
Bo~Stenstr{\"o}m.
\newblock {\em Rings of quotients}.
\newblock Springer-Verlag, New York, 1975.
\newblock Die Grundlehren der Mathematischen Wissenschaften, Band 217, An
  introduction to methods of ring theory.

\bibitem[Ste14]{Ste14}
Greg Stevenson.
\newblock Derived categories of absolutely flat rings.
\newblock {\em Homology Homotopy Appl.}, 16(2):45--64, 2014.

\bibitem[{\v{S}}{\v{t}}o14]{St14}
Jan {\v{S}}{\v{t}}ov{\'{\i}}{\v{c}}ek.
\newblock Exact model categories, approximation theory, and cohomology of
  quasi-coherent sheaves.
\newblock In David~J. Benson, Henning Krause, and Andrzej Skowro{\'n}ski,
  editors, {\em Advances in Representation Theory of Algebras ({C}onf. {ICRA}
  {B}ielefeld, {G}ermany, 8-17 {A}ugust, 2012)}, EMS Series of Congress
  Reports, pages 297--367. EMS Publishing House, Z{\"u}rich, 2014.

\bibitem[Tho97]{Th97}
R.~W. Thomason.
\newblock The classification of triangulated subcategories.
\newblock {\em Compositio Math.}, 105(1):1--27, 1997.

\end{thebibliography}

\end{document}